\documentclass[12pt,a4paper]{article}
\usepackage{amsmath, amsthm, amssymb}
\usepackage[margin=3cm]{geometry}
\usepackage{graphicx}
\usepackage[font=small,format=plain,indention=1ex]{caption,subfig}
\usepackage{color}
\usepackage{tikz}
\usetikzlibrary{arrows}
\usepackage{pgf}
\usepackage{placeins}
\usepackage{mathtools} 

\numberwithin{equation}{section}
\numberwithin{figure}{section}

\newtheorem{thm}{Theorem}[section]

\newtheorem{lem}[thm]{Lemma}

\newtheorem{prop}[thm]{Proposition}

\newcommand{\argmin}{\operatorname{argmin}}
\newcommand{\Rset}{\mathbb{R}}
\newcommand{\mc}{\mathcal}
\newcommand{\nl}{\left\|}
\newcommand{\nr}{\right\|}
\newcommand{\norm}[1]{\left\Vert #1 \right\Vert}

\newcommand{\bfx}{\mathbf{x}}
\newcommand{\bfy}{\mathbf{y}}

\newcommand{\FPi}{F_{\rm 4Pi}}
\newcommand{\pconf}{h}
\newcommand{\pfourPi}{p}

\newcommand{\Lip}{L}
\newcommand{\Xspace}{\mc{X}}
\newcommand{\Yspace}{\mc{Y}}

\begin{document}

\title{The Iteratively Regularized Gau{\ss}-Newton Method with Convex Constraints 
and Applications in 4Pi-Microscopy}
\author{Robert St\"uck, Martin Burger and Thorsten Hohage}
\date{\empty}

\maketitle
\abstract{This paper is concerned with the numerical solution of nonlinear ill-posed operator
equations involving convex constraints. We study a Newton-type method which
consists in applying linear Tikhonov regularization with convex constraints to the Newton
equations in each iteration step. Convergence of this iterative regularization method is analyzed 
if both the operator and the right hand side are given with errors and all error levels
tend to zero. Our study has been motivated by the joint estimation of object and phase
in 4Pi microscopy, which leads to a semi-blind deconvolution problem with non\-ne\-ga\-ti\-vi\-ty
constraints. The performance of the proposed algorithm is illustrated both for  simulated and for three-dimensional experimental data.}

\section{Introduction}
In this paper we present and analyze a Newton-type regularization method for nonlinear ill-posed
operator equations with convex constraints. More specifically, let 
$\Xspace$ and $\Yspace$ be Hilbert spaces, $\mc{C}\subset\Xspace$ a nonempty, closed convex set, 
and $F:\mc{C}\to \Yspace$ a ``forward''  operator, which we assume to be Gateaux differentiable. 
We consider the inverse problem of reconstructing $x^{\dagger}$ in the operator equation 
\begin{eqnarray}\label{eq-op}
F(x^{\dagger}) = g,\qquad x^{\dagger}\in\mc{C}
\end{eqnarray}
if only noisy versions of both $F$ and $g$ are given. Moreover, we aim to 
prove convergence of such reconstructions as the noise levels tend to zero. 

An inverse problem for which it is particularly important to properly incorporate a convex 
constraint into the inversion scheme arises in a confocal fluorescence microscopy technique (cf.\ \cite{Pawley})
called 4Pi microscopy. 
This technique was suggested and developed by Hell et.al.\ \cite{hell-fund-imp,Hell-4Pi} 
and allows for a substantial enhancement of resolution using interference of two laser beams in 
the microscopic focus and/or interference of fluorescence photons on the detector. 
In standard confocal microscopy the relation between the unknown fluorescent
marker density $f\in L^2(\mathbb{R}^3)$ of the specimen 
and the measured intensity $g$ is given
by a convolution with a \emph{point spread function} (psf) 
$\pconf \in L^1(\Rset^3)$, which is often modeled as a Gaussian function:
\begin{align}\label{eq:conf_conv}
g(\bfx) = \int \pconf (\bfx-\bfy) f(\bfy)\,d\bfy
\end{align}
The width of $\pconf $ is typically much larger along the 
so-called optical axis (which we assume to be the $x_3$-axis) than in 
directions perpendicular to the optical axis. 

4Pi microscopy allows an increase of resolution along the optical axis by
a factor of 3--7 using interference of coherent photons through two opposing 
objective lenses. Here the psf is no longer spatially invariant in general, 
but depends on the relative phase $\phi(\bfx)$ of the interfering photons,
which has to be recovered together with the fluorophore density $f$ in general
since it depends on the refractive index of the specimen which is unknown. 
The imaging process can be modeled by an operator equation 
$\FPi(f,\phi) = g$ with a forward operator of the form 
\begin{align}\label{eq-4Pi-op}
\FPi(f,\phi)(\bfx):=\int \pfourPi(\mathbf{y-x},\phi(\bfx))f(\bfy)d\bfy\,.
\end{align}
Note that $F$ is nonlinear in $\phi$ and that $f\mapsto F(f,\phi)$ is not
a convolution operator in general. As a density, $f$ has to be nonnegative. 
Therefore, we have the convex constraint 
$(f,\phi)\in \mathcal{C}$ with $\mathcal{C}:=\{(f,\phi):f\geq 0\}$.
A simple frequently used model for the 4Pi-psf (cf.,~e.g.,~\cite{Baddeley}) 
is given by 
\begin{align}\label{eq-psf-BCC}
\pfourPi(\bfx,\varphi) \approx \pconf (\bfx)
\cos^{n} \left(cx_{3} + \frac{\varphi}{2}\right),
\end{align}
where $\pconf $ is the psf of the corresponding confocal microscope,
and the cosine term represents the interference pattern for different types of 4Pi-microscopes corresponding to $n=2,4$, respectively (see Fig.~\ref{fig-4Pi-psf}). 
So far reconstruction of $f$ in commercially available 4Pi microscopes is done
by  standard deconvolution software assuming the relative phase function $\phi$ 
to be constant. Although spatial variations of $\phi$ can approximately 
be avoided experimentally
in some situations, the assumption that $\phi$ is constant imposes severe limitations
on the applicability and reliability of 4Pi microscopy. Therefore, it is of great
interest to develop algorithms for the solution of the convexly constrained
nonlinear inverse problem to recover both the object function $f$
and the relative phase function $\phi$ from the data $g$.  
%
\begin{figure}[htb]%
\def\fsca{0.133} 
\begin{center}%
\subfloat[$\varphi = 0$\label{fig-psfBCC-phi0}]{%
\begin{tikzpicture}%
\pgftext[]{\rotatebox{0}{\includegraphics[scale=\fsca]{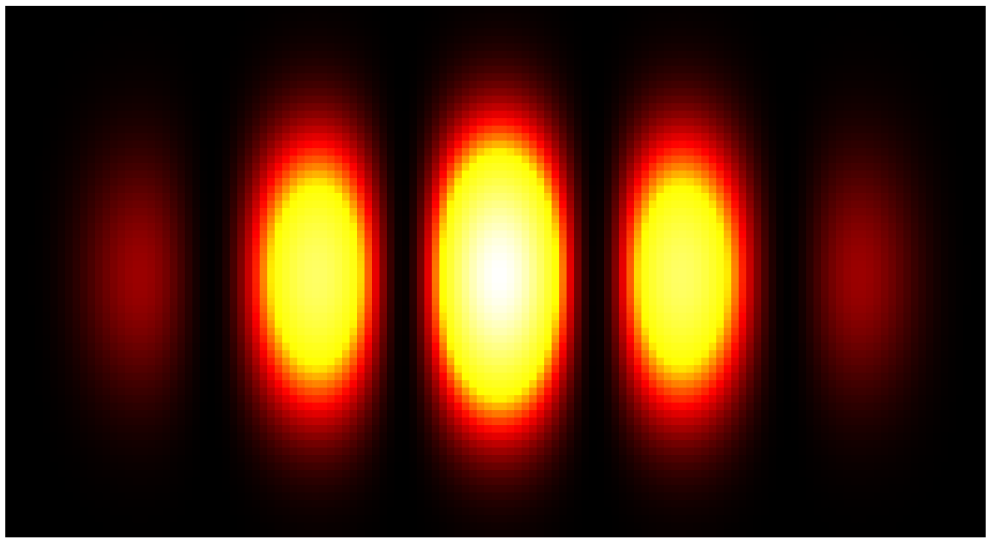}}}%
\draw[color=white,line width=1 pt,->,-latex] plot coordinates{(-1,1) (1,1)}; %
\end{tikzpicture}%
}%
\subfloat[$\varphi = \frac{\pi}{2}$\label{fig-psfBCC-phipi2}]{%
\begin{tikzpicture}%
\pgftext[]{%
\rotatebox{0}{\includegraphics[scale=\fsca]{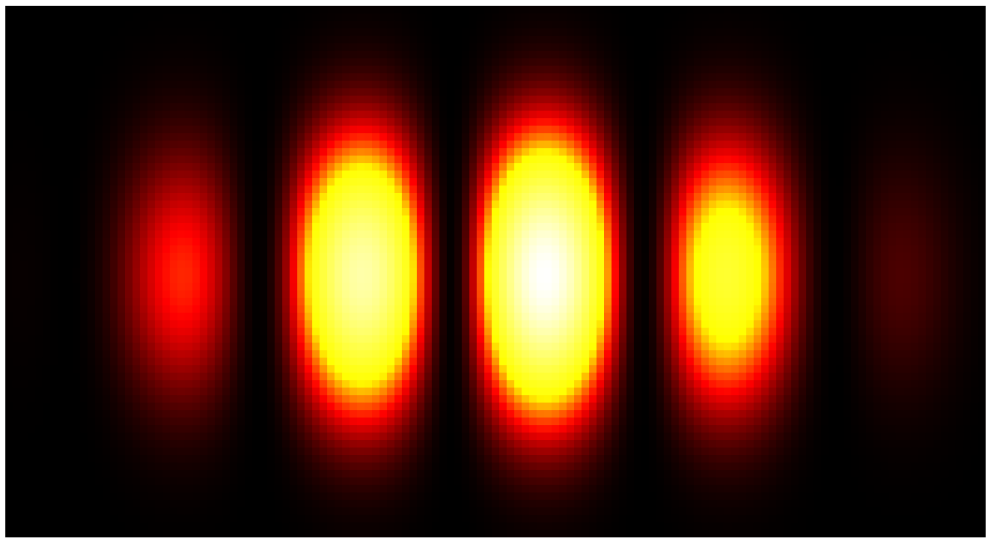}}%
}%
\draw[color=white,line width=1 pt] plot coordinates{(-0,1.3) (0.1,1.3)}; %
\end{tikzpicture}%
}%
\subfloat[$\varphi = \pi$\label{fig-psfBCC-phipi}]{%
\begin{tikzpicture}%
\pgftext[]{%
\rotatebox{0}{\includegraphics[scale=\fsca]{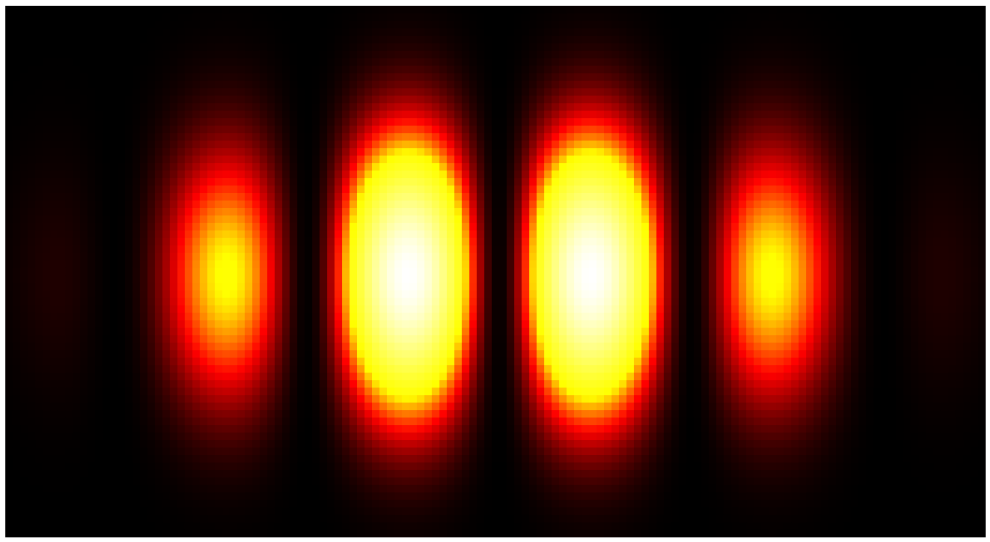}}%
}%
\draw[color=white,line width=1 pt] plot coordinates{(-0,1.3) (0.1,1.3)}; %
\end{tikzpicture}%
}\\%
\subfloat[$\varphi = 0$\label{fig-intr-4Pi-psf-phi0}]{%
\begin{tikzpicture}%
\pgftext[]{\rotatebox{0}{\includegraphics[scale=\fsca]{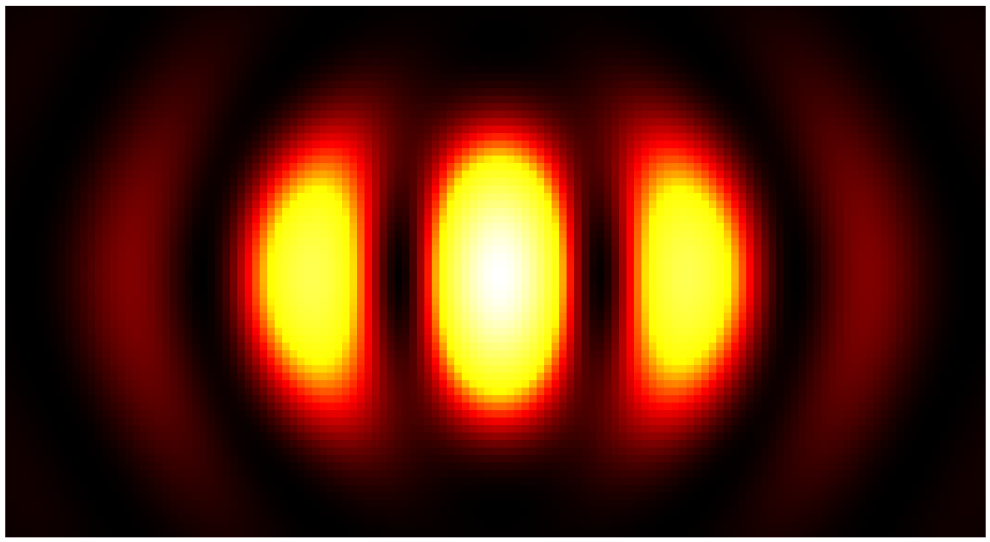}}}%
\draw[color=white,line width=1 pt,->,-latex] plot coordinates{(-1,1) (1,1)}; %
\end{tikzpicture}%
}%
\subfloat[$\varphi = \frac{\pi}{2}$\label{fig-intr-4Pi-psf-phipi2}]{%
\begin{tikzpicture}%
\pgftext[]{%
\rotatebox{0}{\includegraphics[scale=\fsca]{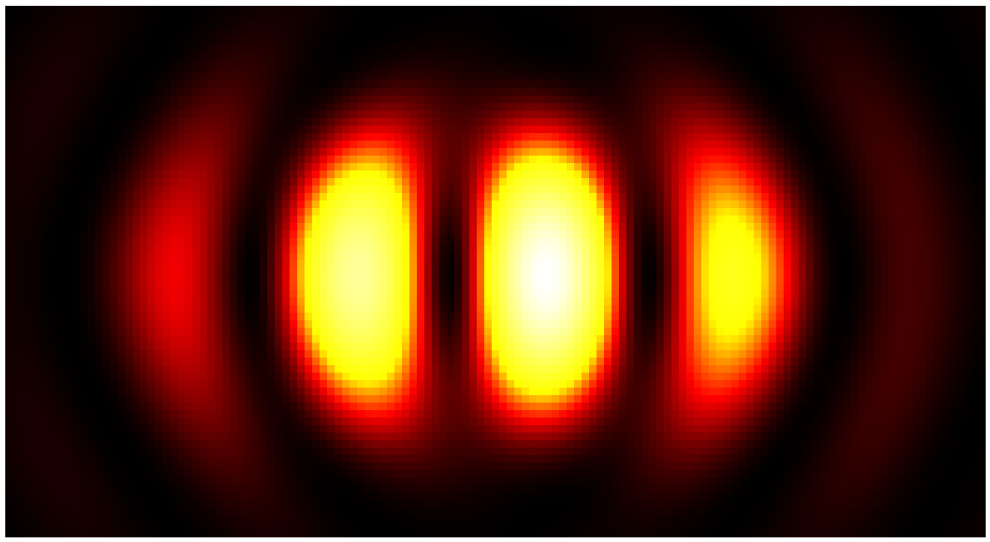}}%
}%
\draw[color=white,line width=1 pt] plot coordinates{(-0,1.3) (0.1,1.3)}; %
\end{tikzpicture}%
}%
\subfloat[$\varphi = \pi$\label{fig-intr-4Pi-psf-phipi}]{%
\begin{tikzpicture}%
\pgftext[]{%
\rotatebox{0}{\includegraphics[scale=\fsca]{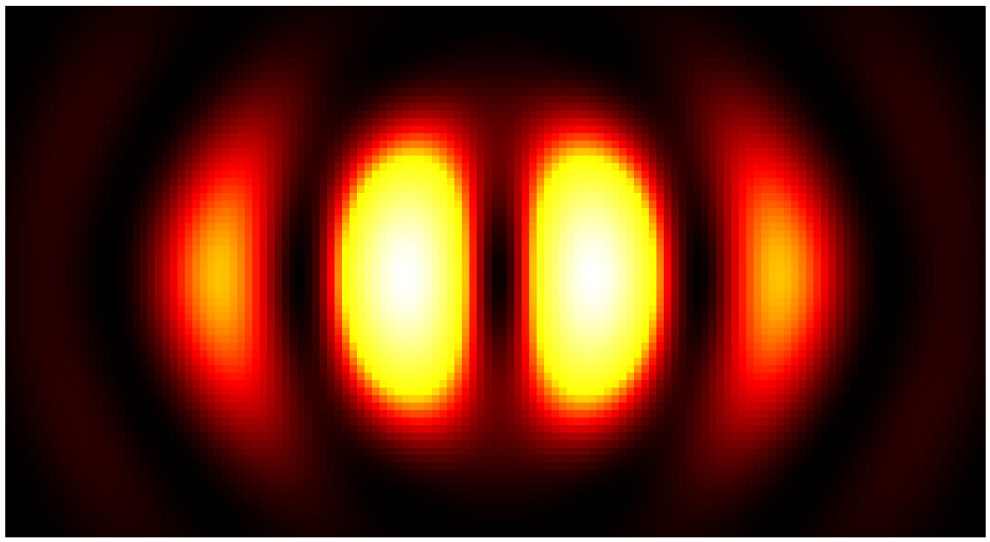}}%
}%
\draw[color=white,line width=1 pt] plot coordinates{(-0,1.3) (0.1,1.3)}; %
\end{tikzpicture}%
}%
\end{center}%
\caption{The top line shows the psf of a 4Pi microscope modeled by \eqref{eq-psf-BCC} for relative phases 
$\varphi=0,\frac{\pi}{2},\pi$ on the plane containing the optical axis, which is indicated by the 
white arrow. The bottom line shows the more accurate model \eqref{eq:accurate4PiPSF}.
}\label{fig-4Pi-psf}%
\end{figure}%

To this end we propose and analyze the following constrained version of the 
iteratively regularized Gau{\ss}-Newton method (IRGNM). We assume that both 
the right hand side $g$ in the operator equation \eqref{eq-op} and the 
operator $F$ are
only given approximately with errors by $g_{\delta}$ and $F_{\delta}$, 
respectively. Error bounds will be specified in the next section. 
Given some initial guess $x_0\in\mc{C}$, we consider the iteration   
\begin{subequations}\label{eqs:IRGNM}
\begin{align}\label{eq-IRGNMwc}
x_{n+1} = \underset{x\in \mc{C}}{\operatorname{argmin}}\left[
\left\| F_{\delta}'[x_{n}](x-x_{n}) + F_{\delta}(x_{n}) - g_{\delta} \right \|^2 + \alpha_{n} \|x-x_{0}\|^2\right]\,,
\end{align}
$n=0,1,2,\dots$, 
with a sequence of regularization parameters $\alpha_n$ satisfying
\begin{align}\label{eq-alpha}
1\leq\frac{\alpha_{n}}{\alpha_{n+1}}\leq r , \quad \lim_{n\to \infty}\alpha_{n}=0, \quad \alpha_{n}>0 \quad \mbox{for some } r> 1 
\mbox{ and for all }n\in \mathbb{N}_0.
\end{align}
\end{subequations}
In the unconstrained case $\mc{C}=\Xspace$ this reduces to the IRGNM as 
suggested in the original paper by Bakushinski{\u\i} \cite{Baku92}. For 
$\mc{C}\neq \Xspace$ a quadratic minimization problem with convex constraint
has to be solved in each Newton step. 
In \cite{Baku92} convergence rates were shown for H\"older type 
source conditions with exponent $\nu=1$.  In \cite{blasch97, Hohage1997} order 
optimal convergence rates for more general H\"older type and logarithmic 
source conditions were proven. For numerous further references on the 
IRGNM and other iterative regularization methods we
refer to the monographs \cite{BK:04,KNS:08}. 
More recently, Kaltenbacher and Hofmann \cite{kalt2010} proved optimal 
convergence rates of the IRGNM in Banach spaces for general source conditions. 

The convergence result we will present in the next section
(Theorem \ref{thm-main}) takes into account two features, which are
essential for 4Pi reconstructions and are not covered in the literature
so far: First of all, our source condition takes into account the convex 
constraint and is weaker than the corresponding source condition for
the unconstrained case, yielding the same rate of convergence. 
This reflects the observation reported below that projecting 
reconstructions of the unconstrained IRGNM onto $\mc{C}$ does not yield 
competitive results. For linear Tikhonov regularization with
convex constraints we refer to Neubauer \cite{Neub88}
and \cite[section 5.4]{EnglRIP}. Moreover, unlike many other references
on the IRGNM, we also take into account errors in the operator since
they are important in our application: The frequently used 
model \eqref{eq-psf-BCC} for the 4Pi psf is only a first approximation, 
and even the more accurate model based on the evaluation of diffraction
integrals, which we used in our code (see Fig.~\ref{fig-4Pi-psf}
and eq.~\eqref{eq:accurate4PiPSF} below), contains parameters,
which have to be estimated including errors. Other references discussing
the influence of errors in the operator for the IRGNM include 
\cite{BK:04} and \cite{hohage:00}. 





The plan of this paper is as follows: 
Our main convergence result, Theorem \ref{thm-main}, is formulated and 
proved in Section \ref{sec-IRGNMwcc}. Section \ref{sec-application} contains
a more detailed discussion of 4Pi microscopy and the model 
\eqref{eq-4Pi-op}, a comparison with other methods, and 
numerical results both for simulated and experimental data.

\section{IRGNM with Convex Constraints}\label{sec-IRGNMwcc}
\subsection{Formulation of the theorem}
We assume that $F, F_{\delta}:\mc{C}\to \Yspace$ are both Gateaux differentiable
with bounded derivatives $F'[x],F_{\delta}[x]$ for all $x\in\mathcal{C}$ 
and that the following error bounds hold:
\begin{subequations}\label{eqs-noisebounds}
\begin{align}
\label{eq-noise}
\nl g-g_{\delta}\nr &\leq \delta_{g},\\
\norm{F(x^{\dagger})-F_{\delta}(x^{\dagger})}&\leq \delta_{F}\\
\norm{F'[x^{\dagger}]-F'_{\delta}[x^{\dagger}]}&\leq \delta_{F'}\label{eq-operr-deriv}
\end{align}
\end{subequations}
with noise levels $\delta_{g},\delta_{F},\delta_{F'}\geq 0$. 

Further we assume that a source condition of the form
\begin{subequations}\label{eqs-source}
\begin{align}\label{eq-source}
x^{\dagger}&=P_{\mc{C}}(F'[x^{\dagger}]^\ast\omega + x_{0})\quad \text{for some } \omega\in \Yspace \text{ with } \nl \omega\nr\leq \rho
\end{align}
is satisfied where $P_{\mc{C}}: \Xspace\rightarrow \mc{C}$ denotes the 
metric projection onto $\mc{C}$. %
The source condition \eqref{eq-source} corresponds to the one for linear constrained Tikhonov regularization we assume in Lemma \ref{thm-519}, and since $\mc{R}(T^{*})=\mc{R}((T^{*}T)^{1/2})$ for a bounded linear operator $T:\Xspace\rightarrow \Yspace $ (cf.\ \cite[Proposition 2.18]{EnglRIP}) it corresponds to a H\"older-type source condition with exponent $\nu = \frac{1}{2}$. %
%
As the \eqref{eq-source} contains 
the projector $P_{\mc{C}}$, it is less restrictive than in the unconstrained case 
$\mc{C}=\Xspace$. In particular, $x^{\dagger}$ may not be smooth even if $F'[x^{\dagger}]^*$ is smoothing and $x_0$ is smooth. 
%

If $F'[x^{\dagger}]$ is not injective, we further assume that $x^{\dagger}$ 
satisfies
\begin{align}\label{eq:minimum_norm_sol}
x^{\dagger} =  \underset{\{x\in \mc{C}:F'[x^{\dagger}](x-x^{\dagger})=0\}} {\operatorname{argmin}}
\|x-x_0\|\,.
\end{align}
\end{subequations}
Obviously, this condition is empty if $F'[x^{\dagger}]$ is injective. 
Moreover, if for any $v_0\in N(F'[x^{\dagger}])$ there
exists a differentiable curve $v:[0,\epsilon)\to \mathcal{C}$ with
$v(0)=x^{\dagger}$, $v'(0)=v_0$ and $F(v(t))=g$ for all $t$ 
(see e.g. \cite{HH:09} for a problem where this condition is satisfied), then
it is easy to see that \eqref{eq:minimum_norm_sol} follows from 
$x^{\dagger}= {\operatorname{argmin}}_{\{x\in \mc{C}:F(x)=g\}}\|x-x_0\|$.

As nonlinearity condition on the operator $F_{\delta}$ we only need to assume 
that for some $\gamma >0$ there exists a Lipschitz constant $\Lip>0$ such that
\begin{eqnarray}\label{eq-Lipschitz}
\left\|F_{\delta}'[x]-F_{\delta}'[x^{\dagger}]\right\| \leq \Lip \norm{x-x^{\dagger}}\qquad \mbox{for all } x\in\mc{C} \text{ with }\norm{x-x^{\dagger}}\leq \gamma .
\end{eqnarray}
We can now formulate our main convergence theorem:
\begin{thm}\label{thm-main} 
Assume that \eqref{eq-op} and \eqref{eqs-noisebounds}--\eqref{eq-Lipschitz} 
are satisfied with $\rho$ is sufficiently small, set 
\begin{align}\label{eq-allerr}
\overline{\delta}:=\max(\delta_{g}+\delta_{F},\delta_{F'}^{2}),
\end{align}
and consider the sequence $(x_n)$ defined by \eqref{eqs:IRGNM}.\\
Then the iterates satisfy $\norm{x_{n}-x^{\dagger}}\leq\gamma$ and in the noise free case $\overline{\delta} = 0$ we have
\begin{align}\label{eq-noisefreerate}
 \nl x_{n}-x^{\dagger}\nr=\mc{O}(\sqrt{\alpha_{n}}),
 \qquad n\to\infty.
 \end{align}
For $\overline{\delta} > 0$ assume that a stopping index $N$ is chosen such that
\begin{eqnarray}\label{eq-stopping}
\alpha_{N}<\eta\overline{\delta}\leq\alpha_{n},\quad 0\leq n< N
\end{eqnarray}
with some constant $\eta > 0$ sufficiently large. Then the error of the final approximation fulfills
\begin{eqnarray}\label{eq-noisyrate}
 \nl x_{N} - x^{\dagger}\nr = \mc{O}\left(\sqrt{\overline{\delta}}\right),
 \qquad \overline{\delta}\to 0\,.
 \end{eqnarray}
\end{thm}

\subsection{Proof of the theorem}
Note that if $F=F_{\delta}=T:\Xspace\rightarrow \Yspace$ is linear and bounded, 
then \eqref{eq-IRGNMwc} reduces to linear constrained Tikhonov regularization 
\begin{align}\label{eq-lintikh}
x_{\alpha}:=\underset{x\in\mc{C}}{\argmin}
\left[\norm{Tx-g_{\delta}}^{2} + \alpha\norm{x-x_{0}}^{2}\right]
\end{align}
for a sequence of regularization parameter $\alpha=\alpha_n$. 
We first recall the stability and approximation properties in this case 
since they will be needed later in the proof.  
\begin{lem}\label{thm-519}
\begin{enumerate}
\item
If 
\[\overline{x}_{\alpha}:=\underset{x\in\mc{C}}{\argmin}
\left[\norm{Tx-\overline{g}_{\delta}}^{2} + \alpha\norm{x-x_{0}}^{2}\right]
\]
for some $\overline{g}_{\delta}\in \Yspace$, then
\begin{align}\label{eq:stabil}
\|x_{\alpha}-\overline{x}_{\alpha}\|\leq \frac{\|g_{\delta}-\overline{g}_{\delta}\|}
{\alpha}\,.
\end{align}
\item
Let $g=g^{\delta}\in T(\mc{C})$ and $x_0\in \mc{C}$, and assume that the 
best-approximate-solution 
$x^{\dagger}_{\mc{C}}:=\argmin_{\{x\in\mc{C}:Tx=g \}}\|x-x_0\|$
satisfies the source condition
\begin{align}\label{eq-source519}
x^{\dagger}_{\mc{C}} = P_{\mc{C}}(T^{*}\omega+x_0)
\end{align}
for some $\omega\in \Yspace$. Then 
\begin{align}\label{eq:linear_approx}
\norm{x_{\alpha}-x^{\dagger}_{\mc{C}}} = \sqrt{\alpha}\norm{\omega}\quad \text{and}\quad \norm{Tx_{\alpha}-g} = \alpha\norm{\omega}.
\end{align}
\end{enumerate}
\end{lem}
\begin{proof}
In \eqref{eq:stabil} the special case $x_0=0$ is proved in 
\cite[Theorem 5.16]{EnglRIP}. The general case can be reduced
to this special case by the substitution of variables $z=x-x_0$ since 
\begin{align}\label{eq-lintikh-trans}
x_{\alpha}- x_0 = \underset{z\in\mc{C}-x_0}{\argmin}
\left[\norm{Tz + Tx_0-g_{\delta}}^{2} + \alpha\norm{z}^{2}\right]\,.
\end{align}
\eqref{eq:linear_approx} can be reduced to the case $x_0=0$, which
is covered by \cite[Theorem 5.19]{EnglRIP}, by the same substitution
of variables and the identity
$P_{\mc{C}-x_0}(T^\ast\omega) = P_{\mc{C}}(T^\ast\omega + x_{0})-x_0$.
\end{proof}
%
%
%
%
Next we need a stability estimate with respect to perturbations of the
operators, i.e.\ an estimate on the difference of
\begin{eqnarray}\label{eq-prop-funct}
 x_i &:=&  \underset{x\in\mc{C}}{\argmin} 
 \left[\left\| T_{i}(x - \tilde x) \right \|^2
 + \alpha \left\| x-x_{0}\right\|^2\right],\quad i\in\{1,2\},
 \end{eqnarray}
where $T_1,T_2:\Xspace\to \Yspace$ are bounded linear operators 
and $\alpha>0$. Using the optimality conditions for the minimizers of 
\eqref{eq-prop-funct}, a straightforward computation gives an estimate
of the form
\begin{align}
\nl x_{1}-x_{2}\nr \leq \frac{c}{\alpha}\nl T_{1}-T_{2}\nr.
\end{align}
This simple estimate is not sufficient for our purposes, however. 
The following proposition shows that under a source condition we can obtain
an improved estimate with a constant independent of $\alpha$:
\begin{prop}\label{prop-hyp}
Let $x_1$ and $x_2$ be defined by \eqref{eq-prop-funct}.
Moreover, let the source condition
 \begin{eqnarray}\label{eq-prop-source}
\tilde x = P_{\mc{C}}(T_{2}^\ast\omega + x_{0})
\end{eqnarray}
hold for some $\omega\in \Yspace$ and let 
$\tilde{x} = \argmin_{\{x\in\mc{C}:T_2x=T_2\tilde{x} \}}\|x-x_0\|$. 
Then the distance of $x_{1}$ and $x_{2}$ is bounded by
 \[ \nl x_{1}-x_{2}\nr \leq \sqrt{\frac{3}{2}}\norm{\omega}\nl T_{1}-T_{2}\nr .\]
 \end{prop}
 
 \begin{proof}
From Lemma \ref{thm-519}, part 2 we obtain
\begin{subequations}\label{eq-ripres}
\begin{eqnarray}
\nl T_{2}(x_{2} - \tilde x)\nr &\leq& \alpha \nl \omega\nr\\
\nl x_{2} - \tilde x\nr &\leq& \sqrt{\alpha} \nl \omega\nr .
\end{eqnarray}
\end{subequations}
Let $\chi_{\mc{C}}:X\rightarrow \mathbb{R^{+}}$ be the proper, convex and lower semicontinuous functional
\begin{align*}
\chi_{\mc{ C}}(x):=
\begin{cases}
0 & x\in \mc{C}\\
\infty & \mbox{otherwise}
\end{cases},
\end{align*}
and let $q_{i}\in \partial \chi_{\mc{ C}}(x_{i})$, then for $i\in\{1,2\}$ the first order optimality condition for the minimizers $x_{i}$ is given by
\begin{align}\label{eq-start}
T_{i}^\ast T_{i} (x_{i}-\tilde x) + \alpha (x_{i} - x_{0}) + q_{i} = 0.
\end{align}
Subtracting the equations \eqref{eq-start} gives
\begin{eqnarray*}
T_{1}^\ast T_{1} (x_{1} - x_{2}) + \alpha (x_{1}-x_{2}) + (q_{1}-q_{2})= (T_{2}^\ast T_{2} - T_{1}^\ast T_{1})(x_{2}-\tilde x) .
\end{eqnarray*}
Now taking the inner product with $x_{1}-x_{2}$ we obtain
\begin{align}\label{eq-prop-inner}
\begin{aligned}
& \nl T_{1}(x_{1}-x_{2})\nr^2 + \alpha \nl x_{1}-x_{2}\nr^2 +\langle q_{1} - q_{2},x_{1}-x_{2}\rangle =\\
& \langle(T_{2}^\ast - T_{1}^\ast)T_{2}(x_{2}-\tilde x),x_{1}-x_{2}\rangle + \langle(T_{2}-T_{1})(x_{2}-\tilde x),T_{1}(x_{1}-x_{2})\rangle .
\end{aligned}
\end{align}
The right hand side can be estimated with help of Young's inequality
\begin{align}\label{eq-young}
&\langle(T_{2}^\ast - T_{1}^\ast)T_{2}(x_{2}-\tilde x),x_{1}-x_{2}\rangle + \langle(T_{2}-T_{1})(x_{2}-\tilde x),T_{1}(x_{1}-x_{2})\rangle\\
& \leq \frac{1}{2\alpha}\norm{T_{1}-T_{2}}^{2}\norm{T_{2}(x_{2}-\tilde{x})}^{2} + \frac{\alpha}{2}\norm{x_{1}-x_{2}}^{2} + \frac{1}{4}\norm{T_{1}-T_{2}}^{2}\norm{x_{2}-\tilde{x}}^{2}+\norm{T_{1}(x_{1}-x_{2})}^{2}.\nonumber
\end{align}
Using $\langle q_{1} - q_{2},x_{1}-x_{2}\rangle \geq 0$ 
(see, e.g., \ \cite[Section 9.6.1, Theorem 1]{Evans}),  \eqref{eq-young} 
and the inequalities \eqref{eq-ripres} the assertion follows from
\begin{align*}
\frac{\alpha}{2} \nl x_{1}-x_{2}\nr^2 &\leq \frac{1}{2\alpha}\nl T_{2}(x_{2}-\tilde x)\nr^2\nl T_{1}-T_{2}\nr^2 + \frac{1}{4}\nl x_{2}-\tilde x\nr^2\nl T_{1}-T_{2}\nr^2\\
&\leq\frac{3}{4}\alpha \nl\omega\nr^2\nl T_{1}-T_{2}\nr^2.
\end{align*}
\end{proof}
Now we are able to formulate a recursive error estimate for the IRGNM with closed convex constraint.
\begin{lem}\label{lem-recursive}
Assume that  \eqref{eq-op} and \eqref{eqs:IRGNM}--\eqref{eq-Lipschitz} are satisfied 
and that $x_{n}\in \mc{C}$ with $\norm{x_{n}-x^{\dagger}}\leq\gamma$, then the error $e_{n}:=x_{n}-x^{\dagger}$ satisfies
\begin{align}\label{eq-lem-recursive}
\nl e_{n+1}\nr \leq \frac{1}{\sqrt{\alpha_{n}}}\frac{\Lip}{2}\nl e_{n}\nr^2 + \sqrt{\frac{3}{2}}\rho \Lip\nl e_{n}\nr + \frac{1}{\sqrt{\alpha_{n}}} \left(\delta_{g} + \delta_{F}\right) + \sqrt{\alpha_{n}}\rho + \sqrt{\frac{3}{2}}\rho\delta_{F'}.
\end{align}
\end{lem}

\begin{proof}
At first  we note that one can express the noisy data as $g_{\delta} = F_{\delta}(x^{\dagger}) + \xi + \epsilon$, with $\norm{\xi}\leq \delta_{F}$ and $\nl\epsilon\nr\leq \delta_{g}$. Further since $x_{n}\in\mc{C}$ and $F_{\delta}$ is Gateaux differentiable with derivatives that fulfill condition \eqref{eq-Lipschitz}, we can express $F_{\delta}(x^{\dagger})$ in a Taylor series
\begin{align}
F_{\delta}(x^{\dagger}) = F_{\delta}(x_{n}) + F_{\delta}'[x_{n}](x^{\dagger}-x_{n}) +r(x^{\dagger}-x_{n}),
\end{align}
where 
\begin{align}\label{eq-remainder}
\nl r(x^{\dagger}-x_{n})\nr\leq \frac{\Lip}{2}\nl x^{\dagger}-x_{n}\nr^{2} .
\end{align}
Thus we can rewrite the IRGNM functional \eqref{eq-IRGNMwc} of the $n$-th iteration step, defining $T_{n}:=F_{\delta}'[x_{n}]$, as
\begin{align}
& \left\| T_{n}x - \left(T_{n}x_{n}- F_{\delta}(x_{n})+ g_{\delta} \right)\right \|^2
 + \alpha_{n} \left\| x-x_{0}\right\|^2\\
 &=  \left\| T_{n}(x-x^{\dagger}) - r(x^{\dagger}-x_{n})-\xi - \epsilon\right \|^2
 + \alpha_{n} \left\| x-x_{0}\right\|^2 .
\end{align}
Now we can decompose the distance of the solution $x_{n+1}$ of the $(n+1)$-th iteration to the exact solution $x^{\dagger}$ using the triangle inequality
\begin{align}\label{eq-rec1}
\left\|x_{n+1}-x^{\dagger}\right\|\leq \left\|x_{n+1}-x_{\alpha_{n},n}\right\| +\left\|x_{\alpha_{n},n} - x_{\alpha_{n}}\right\| +\left\|x_{\alpha_{n}}-x^{\dagger}\right\| 
\end{align}
with
\begin{subequations}
\begin{align*}
x_{n+1} &\hphantom{:}= \underset{x\in\mathcal{C}}{\argmin} \left[\left\|T_{n}(x - x^\dagger) - r_{n}(x^{\dagger}-x_{n}) - \xi - \epsilon\right\|^2 + \alpha_{n}\left\|x-x_{0}\right\|^2\right]\\
x_{\alpha_{n},n}&:= \underset{x\in\mathcal{C}}{\argmin}\left[ \left\|T_{n}(x-x^\dagger) \right\|^2 + \alpha_{n}\left\|x-x_{0}\right\|^2\right]\\
x_{\alpha_{n}} &:= \underset{x\in\mathcal{C}}{\argmin}\left[ \left\|F'[x^{\dagger}](x-x^{\dagger}) \right\|^2 + \alpha_{n}\left\|x-x_{0}\right\|^2\right]  .
\end{align*}
\end{subequations}
It follows from Lemma \ref{thm-519}, part 1 that
\begin{align}
\left\|x_{n+1} - x_{\alpha_{n},n}\right\| 
\leq \frac{\left\|r_{n}(x^{\dagger}-x_{n}) + \xi + \epsilon\right\|}{\sqrt{\alpha_{n}}}.
\end{align}
With \eqref{eq-remainder}, $\norm{\xi}\leq\delta_{F}$ and $\nl\epsilon\nr\leq \delta_{g}$ we obtain
\begin{align}\label{eq-rec2}
\left\|x_{n+1} - x_{\alpha_{n},n}\right\| \leq \frac{1}{\sqrt{\alpha_{n}}}\left(\frac{\Lip}{2}\left\|x_{n}-x^{\dagger}\right\|^2+ \delta_{F} + \delta_{g}\right).
\end{align}
The second term in \eqref{eq-rec1} can be estimated using Proposition \ref{prop-hyp} with $T_{1}=T_{n}$, $T_{2}=F'[x^{\dagger}]$ and $\tilde x = x^{\dagger}$, which gives
\begin{align}
\left\|x_{\alpha_{n},n} - x_{\alpha_{n}}\right\|&\leq \sqrt{\frac{3}{2}}\rho \norm{T_{n}-F'[x^{\dagger}]}\nonumber\\
&\leq \sqrt{\frac{3}{2}}\rho \left(\norm{T_{n}-F'_{\delta}[x^{\dagger}]}+\norm{F'_{\delta}[x^{\dagger}]-F'[x^{\dagger}]}\right)\nonumber\\
&\leq \sqrt{\frac{3}{2}}\rho \left(\Lip\norm{x_{n}-x^{\dagger}} + \delta_{F'}\right),\label{eq-rec3}
\end{align}
where we used \eqref{eq-Lipschitz} and \eqref{eq-operr-deriv} to obtain the last inequality of \eqref{eq-rec3}.
For the third term in \eqref{eq-rec1} we again use Lemma \ref{thm-519}, part 2
to obtain
\begin{align}\label{eq-rec4}
\left\|x_{\alpha_{n}}-x^{\dagger}\right\| \leq \sqrt{\alpha_{n}}\rho  .
\end{align}
Combining \eqref{eq-rec1}, \eqref{eq-rec2}, \eqref{eq-rec3} and \eqref{eq-rec4} gives the assertion.
\end{proof}
Estimate \eqref{eq-lem-recursive} is of the form used in \cite{blasch97} and \cite{Hohage1997}, so we can use a similar proof now to obtain the main result.
%
%
\begin{proof}[Proof of Theorem \ref{thm-main}]
It follows from Lemma \ref{lem-recursive} that the quantities $\Theta_{n}:=\frac{\nl e_{n}\nr}{\sqrt{\alpha_{n}}}$ fulfill the inequality
\begin{eqnarray}\label{eq-recursive}
\Theta_{n+1} \leq a + b\Theta_{n} + c\Theta_{n}^2
\end{eqnarray}
with $a:=\sqrt{r}\rho$ for $\overline{\delta} = 0$ and $a:=\sqrt{r}(\rho +\frac{\delta_{g} + \delta_{F}}{\eta\overline{\delta}} + \sqrt{\frac{3}{2}}\rho\frac{\delta_{F'}}{\sqrt{\eta\overline{\delta}}})$ for $\overline{\delta}>0$, $b := \sqrt{\frac{3}{2}r}\rho \Lip $ and $c:= \sqrt{r}\frac{\Lip}{2}$. Let $t_{1}$ and $t_{2}$ be solutions to the fixed point equation $a + bt + ct^2 = t$, i.e.
\begin{eqnarray}
t_{1}:= \frac{2a}{1-b + \sqrt{(1-b)^2-4ac}}\qquad t_{2}:=\frac{1-b + \sqrt{(1-b)^2-4ac}}{2c},
\end{eqnarray}
let the stopping index $N\leq \infty$ be given by \eqref{eq-stopping} and define $C_{\Theta}:=\max(\Theta_{0},t_{1})$. We will show by induction that
\begin{eqnarray}\label{eq-Theta}
\Theta_{n}\leq C_{\Theta} 
\end{eqnarray}
for $0\leq n \leq N$ if
\begin{subequations}\label{eqs-ind}
\begin{eqnarray}
b+2\sqrt{ac} < 1,\label{eq-ind1}\\
\Theta_{0}\leq t_{2},\label{eq-ind2}\\
\sqrt{\alpha_{0}}C_{\Theta}\leq\gamma.\label{eq-ind3}%
\end{eqnarray}
\end{subequations}
Conditions \eqref{eqs-ind} are satisfied if $\rho$ is sufficiently small and $\eta$ sufficiently large. For $n=0$ \eqref{eq-Theta} is true by the definition of $C_{\Theta}$. Assume that \eqref{eq-Theta} is true for some $k<N$, 
then by \eqref{eq-ind3} and Lemma \ref{lem-recursive}, \eqref{eq-recursive} is true for $n=k$. Condition \eqref{eq-ind1} assures that $t_{1},t_{2}\in \Rset$ and $t_{1}<t_{2}$, and by \eqref{eq-Theta} one has $0\leq\Theta_{k}\leq t_{1}$ or $t_{1}\leq\Theta_{k}\leq\Theta_{0}$. In the first case, since $a,b,c \geq 0$, we obtain
\begin{eqnarray}
\Theta_{k+1} \leq a + b\Theta_{k} + c\Theta_{k}^2\leq a+bt_{1}+ct_{1}^2=t_{1}.
\end{eqnarray}
In the second  case by \eqref{eq-ind2} and the fact that $a + (b-1)t + ct^2\leq 0$ for $t_{1}\leq t\leq t_{2}$ we obtain
\begin{eqnarray}
\Theta_{k+1}\leq a + b\Theta_{k} + c\Theta_{k}^2\leq\Theta_{k}\leq \Theta_{0}.
\end{eqnarray}
Thus in both cases \eqref{eq-Theta} holds for $n=k+1$ and the induction is complete.

By definition $N=\infty$ for $\overline{\delta} = 0$ and thus \eqref{eq-Theta} directly implies \eqref{eq-noisefreerate}. Using $\alpha_{N} < \eta\overline{\delta}$ by \eqref{eq-stopping} also assertion \eqref{eq-noisyrate} follows directly from \eqref{eq-Theta}.
\end{proof}

\section{Joint reconstruction of object and phase in 4Pi-microscopy}
\label{sec-application}

Before we describe our mathematical model of the 4Pi imaging process precisely, 
let us discuss this technique in some more detail. 
Confocal fluorescence microscopy allows the reconstruction of 
three-dimensional fluorescent marker densities in living cells by
scanning a specimen at a grid of points $\{\bfx_j\in\mathbb{R}^3:j=1,\dots,N\}$. 
Laser light is focused to a small area by objective
lenses, and a pinhole is used to collect only fluorescence photons 
emitted close to the focus $\bfx_j$ (cf.\ \cite{Pawley}). 
The psf $\pconf (\bfx-\bfy)$ in \eqref{eq:conf_conv}
is the probability that a fluorescence photon emitted at $\bfy$ is detected
if the point $\bfx$ is illuminated. Data consist of photon count numbers
$G_j$, $j=1,\dots,N$, which are Poisson distributed random numbers with
mean ${\bf E} G_j = g(\bfx_j)$. 

In 4Pi microscopy the same data model holds true, but $g$ is given by
$g=\FPi(f,\phi)$ with the integral operator \eqref{eq-4Pi-op}.
For its kernel we use the more accurate model
\begin{equation}\label{eq:accurate4PiPSF}
\pfourPi({\bf z},\varphi) = |{\bf E}_1(z)- \exp(i\phi){\bf E}_2(z)|^2 h_{\rm det}({\bf z})
\end{equation}
(see \cite{vicidomini_etal:10}) instead of the simple model
\eqref{eq-psf-BCC}. Here ${\bf E}_{1,2}$ are counterpropagating focal fields 
and $h_{\rm det}$ is the detection psf, for which we used implementations
available under {\tt www.imspector.de}. 

\subsection{forward operator and its derivative}%
We first define appropriate function spaces for the integral operator
$\FPi$ in \eqref{eq-4Pi-op}. 
We assume that $f$ is supported in some cube $\Omega:=\prod_{j=1}^3[-R_j,R_j]$ 
and choose $L^2(\Omega)$ with the standard $L^2$-norm as function space for
the object $f$. We may further assume that $\pfourPi(\cdot,\varphi)$
is supported in some (typically much smaller) cube 
$\prod_{j=1}^3[-r_j,r_j]$ for all $\varphi$ such that $g$ is supported
in $\Omega':=\prod_{j=1}^3[-R_j-r_j,R_j+r_j]$. 

A reason why joint reconstruction of $f$ and $\phi$ from data $g$ 
often works even though the problem is formally underdetermined, is that
$\phi$ can be assumed to be very smooth (often it is even assumed to be
constant). Therefore
we choose the Sobolev space $H^2(\Omega')$ for $\phi$ with norm 
\(\norm{\phi}_{H^{2}(\Omega')}:=\left(\int_{\Omega'}\vert\phi\vert^{2}+\vert\nabla\phi\vert^{2}+\vert\Delta\phi\vert^{2}dx\right)^{\frac{1}{2}}\)
to achieve smooth interpolation in areas where no information on $\phi$
is contained in the data. (This is the case, e.g., in areas where $f$ is
constant. But in such areas $\phi$ is irrelevant for the primary goal to recover $f$.)

The data misfit term should reflect the distribution of data errors. 
Since our data are Poisson distributed, a natural data misfit term would be
the negative log-likelihood function, which is given by 
$\sum_{j=1}^N g(\bfx_j) - G_j\log g(\bfx_j)$. We define a piecewise constant
approximation $g^{\delta}\in L^2(\Omega')$ of the data $(G_j)$ and approximate the
negative log-likelihood by a second order Taylor expansion at $G_j$. This leads 
to a weighted $L^2$ space $\Yspace:=  L^2(\Omega',w)$ with
norm $\|g\|_{\Yspace}^2 = \int_{\Omega'} g^2(\bfx) w(\bfx) \,d\bfx$ and weight 
function 
\[
w(\bfx) = \frac{1}{2\max(g^{\delta}(\bfx),c)},\qquad \bfx\in\Omega'\,.
\]
Here $c>0$ is a small constant avoiding division by zero. As usually multiple 
weaker sources contribute to the data noise, a suitable choice of $c$ 
is the background noise level. 
Better approximations to the Poisson log-likelihood can be achieved by 
taking a Taylor expansion at $g_n = \FPi(f_n,\phi_n)$ and iterating in 
a sequential quadratic programming manner, 
but for the count rates in our experimental data this did not lead to a noticible
improvement. 

In summary, the precise definition of our forward operator is as follows:
\begin{align}\label{eq-defiFpi_precise}
\begin{aligned}
&\FPi: L^2(\Omega)\times H^2(\Omega')\longrightarrow L^2(\Omega',w)\\
&\left(\FPi(f,\phi)\right)(\bfx) := 
\int_{\Omega} \pfourPi(\bfx-\bfy,\phi(\bfx)) f(\bfy)\,d\bfy,\qquad 
\bfx\in\Omega'\,. 
\end{aligned}
\end{align}
Note that $\FPi$ does not change if $p(\cdot,\varphi)$ is replaced by its periodic
extension with period cell $\Omega'$. 

\begin{lem}\label{lem-frechderiv}
If $\pfourPi:\prod_{j=1}^3 (\mathbb{R}/(2(R_j+r_j)\mathbb{Z})) 
\times (\mathbb{R}/\pi\mathbb{Z})\to\mathbb{R}$ 
is continuous and continuously differentiable with respect to 
its last argument, then the operator $\FPi$ defined in 
\eqref{eq-defiFpi_precise} is Fr\'echet differentiable on $\Xspace$ with
\begin{align}\label{eq-frechet}
\FPi'[f,\phi](h_{f}, h_{\phi})(\bfx)=\int_\Omega \left\{ \pfourPi(\bfy -\bfx,\phi(\bfx))h_f(\bfy)  \vphantom{\frac{\partial p}{\partial \phi}} +\frac{\partial p}{\partial \phi}(\bfy - \bfx,\phi(\bfx))f(\bfy)h_\phi(\bfx) \right\} d\bfy,
\end{align}
and the adjoint of 
$F'[f,\phi]:L^2(\Omega)\times H^2(\Omega) \to L^2(\Omega',w)$ is given by 
\begin{align}\label{eq-invprob-adjoint}
\FPi'[f,\phi]^{*}g &= 
\left(\begin{array}{c}
\int_{\Omega'} \pfourPi(\cdot - \bfx,\phi(\bfx))
g(\bfx) w(\bfx) d\bfx\\
 j^{*}\left(g w \int_\Omega \frac{\partial \pfourPi}{\partial \phi}(\cdot-\bfy,
\phi(\cdot))f(\bfy) d\bfy\right)
\end{array}\right)
\end{align}
where $j:H^{2}(\Omega')\hookrightarrow L^{2}(\Omega')$ is the embedding
operator. Moreover, $\FPi'$ satisfies the Lipschitz condition 
\eqref{eq-Lipschitz}
if $\frac{\partial p}{\partial \phi}$ is uniformly Lipschitz continuous with respect to
its last argument. 
\end{lem}

\begin{proof}[Proof (sketch)]
The Fr\'echet differentiability of $F$ and eq.\ \eqref{eq-frechet} 
follow from a Taylor expansion of the kernel $\pfourPi$ with standard estimates
on the Taylor remainder and the continuity of the embedding 
$H^2(\Omega')\hookrightarrow L^{\infty}(\Omega')$. The adjoint 
$F'[f,\phi]^*_{L^2}$ of the continuous extension $F'[f,\phi]_{L^2}$
of $F'[f,\phi]$ to $L^2(\Omega)\times
L^2(\Omega')$ can be computed by interchanging the order of integration.
Then \eqref{eq-invprob-adjoint} follows from 
$F'[f,\phi]^* = \left(F'[f,\phi]_{L^2}\left(\begin{smallmatrix} I& 0 \\ 0 &j
\end{smallmatrix}\right)\right)^*
= \left(\begin{smallmatrix} I& 0 \\ 0 &j^*
\end{smallmatrix}\right) F'[f,\phi]_{L^2}^*$. 
The statement on Lipschitz continuity
is straightforward.
\end{proof}

The crucial observation for an efficient implementation of $\FPi$ and
$\FPi'$ is that $p$ can be separated into
\[
p(z,\varphi) = \sum_{m=-M}^M \exp(im\varphi) A_m(z)
\]
with $A_m\in L^2(\prod_{j=1}^3 (\mathbb{R}/(2(R_j+r_j)\mathbb{Z})))$.
This was observed by Baddeley et al.\ in \cite{Baddeley} for the approximation
\eqref{eq-psf-BCC} and by Vicidomini et al.\ in \cite{vicidomini_etal:10}
for the model \eqref{eq:accurate4PiPSF}. Hence, 
\[
(\FPi(f,\phi))(\bfx) = \sum_{m=-M}^M \exp(im\phi(\bfx)) \, 
\int_{\Omega'}A_m(\bfx-\bfy) f(\bfy)\, d\bfy\,,\qquad \bfx\in \Omega'\,.
\]
Here $f$ is extended by $0$ in $\Omega'\setminus \Omega$ (zero-padding). 
The convolution integrals can be evaluated efficiently using FFT.
An analogous procedure can be applied for the evaluation of
$F'[f,\phi]$ and its adjoint. 

We approximated the phase $\phi$ using tensor products of Chebychev 
polynomials, for which the Gramian matrix with respect to the $H^2$ inner 
product can be computed explicitly.

\subsection{Implementation and necessity of the nonnegativity constraint}
We solve the constrained quadratic minimization problems 
\begin{align}\label{eq-4pi-min}
(f_{n+1},\phi_{n+1}) &:= \underset{\substack{\mathclap{\vphantom{\sum^{F^{2}}}(f,\phi)\in L^{2}(\Omega)\times H^{2}(\Omega),}\\ f\geq 0 \text{ a.e.}}}{\operatorname{argmin}}\left\| \FPi'[(f_{n},\phi_{n})](f,\phi) -  g_{n}\right \|^2 + \alpha_{n} \|(f,\phi)-(f_0,\phi_0)\|^2,
\end{align}
with 
\[g_{n}:=\FPi'[(f_{n},\phi_{n})]((f_{n},\phi_{n})) - \FPi((f_{n},\phi_{n})) + g_{\delta}
\]
 using the semi-smooth Newton method (cf.\ \cite{Hintermueller2003}). %
In each step of this method an unconstrained, positive definite linear system
has to be solved, which is done by the conjugate gradient method. In Figure \ref{fig-lv-rec} the reconstruction of a fluorophore density and the phase from a 2d-slice of real 4Pi data is shown. %
\begin{figure}[!ht]%
\def\fscb{0.58}
\def\fsca{1} 
\def\optx{2}%
\def\optya{-1.5}%
\def\optl{1.16}%
\def\rectxa{5.1}%
\def\rectxl{1.1}%
\def\rectya{-2.5}%
\def\rectyl{5.2}%
\begin{center}%
\subfloat[slice of 4Pi-data $g_{\delta}$\label{sfig-lvdata}]{%
\begin{tikzpicture}[scale = \fscb]%
\pgftext[]{%
\rotatebox{0}{\includegraphics[scale=\fsca]{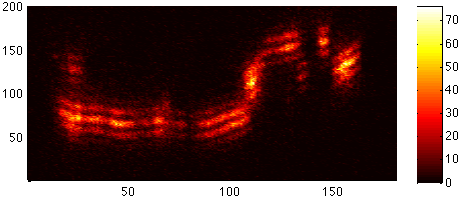}}
}%
\end{tikzpicture}%
}%
\subfloat[psf $\pfourPi(\cdot,0)$\label{sfig-psf}]{%
\begin{tikzpicture}[scale = \fscb]%
\pgftext[]{%
\rotatebox{0}{\includegraphics[scale=\fsca]{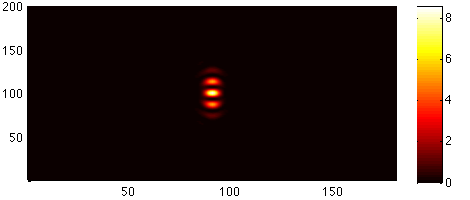}}%
}%
\draw[color=white,line width=1pt,|-|] plot coordinates{(\optx,\optya) (\optx+\optl,\optya)};%
\draw[text=white] {(\optx+\optl*0.5,\optya)} node[below] {{\footnotesize $800$ nm}};%
\draw[color=white,line width=1pt,->, -latex] plot coordinates{(4,2) (4,-1)};%
\fill [white] (\rectxa,\rectya) rectangle (\rectxa+\rectxl,\rectya+\rectyl);%
\end{tikzpicture}%
}\\%
\subfloat[object reconstruction $f_{12}$\label{sfig-obj}]{%
\begin{tikzpicture}[scale = \fscb]%
\pgftext[]{%
\rotatebox{0}{\includegraphics[scale=\fsca]{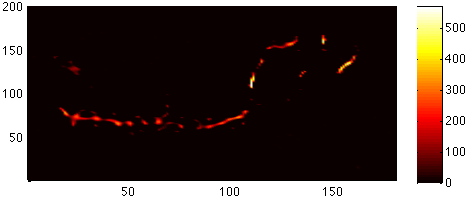}}%
}%
\end{tikzpicture}%
}%
\subfloat[phase reconstruction $\phi_{12}$\label{sfig-phase}]{%
\begin{tikzpicture}[scale = \fscb]%
\pgftext[]{%
\rotatebox{0}{\includegraphics[scale=\fsca]{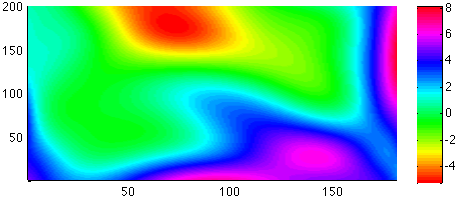}}%
}%
\end{tikzpicture}%
}%
\end{center}%
\caption{Panel \subref{sfig-lvdata} shows a slice of real 3-dimensional 4Pi-data, where in the approximate center the waves were interfering destructively. From these data the reconstructions of object \subref{sfig-obj} and phase \subref{sfig-phase} have been obtained with the constrained IRGNM for $f_{0}=0$ and $\phi_{0}=0$. The reconstruction of the phase reflects the constructive interference on the left and right side of the data and the destructive interference in the center. The modeled psf (for constructive interference) is depicted in panel \subref{sfig-psf} together with indications on the scale and the optical axis (represented by the arrow). To reconstruct the phase we used a basis of polynomials with maximal degree $7$ in each dimension.\label{fig-lv-rec}}%
\end{figure}%
\begin{figure}[!ht]%
\def\fscb{0.58}
\def\fsca{1} 
\def\optx{2}%
\def\optya{-1.5}%
\def\optl{1.16}%
\begin{center}%
\subfloat[object reconstruction $f_{12}^{\rm u}$
\label{sfig-ucobj}]{%
\begin{tikzpicture}[scale = \fscb]%
\pgftext[]{%
\rotatebox{0}{\includegraphics[scale=\fsca]{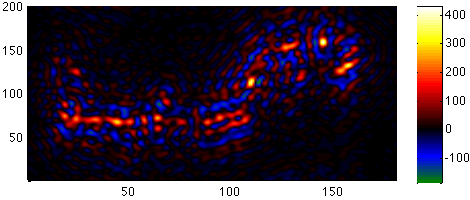}}%
}%
\end{tikzpicture}%
}%
\subfloat[phase reconstruction $\phi_{12}^{\rm u}$
\label{sfig-ucphase}]{%
\begin{tikzpicture}[scale = \fscb]%
\pgftext[]{%
\rotatebox{0}{\includegraphics[scale=\fsca]{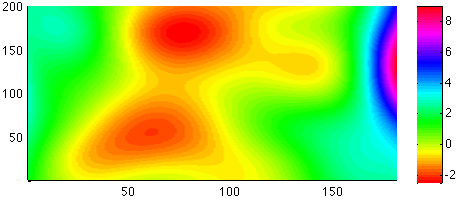}}%
}%
\end{tikzpicture}%
}\\%
\subfloat[object reconstruction $f_{12}^{\rm p}$
\label{sfig-projobj}]{%
\begin{tikzpicture}[scale = \fscb]%
\pgftext[]{%
\rotatebox{0}{\includegraphics[scale=\fsca]{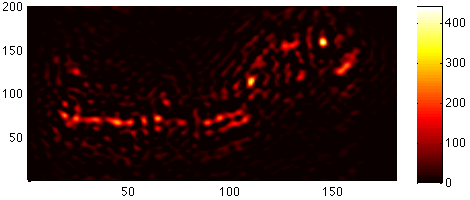}}%
}%
\end{tikzpicture}%
}%
\subfloat[phase reconstruction $\phi_{12}^{\rm p}$
\label{sfig-projphase}]{%
\begin{tikzpicture}[scale = \fscb]%
\pgftext[]{%
\rotatebox{0}{\includegraphics[scale=\fsca]{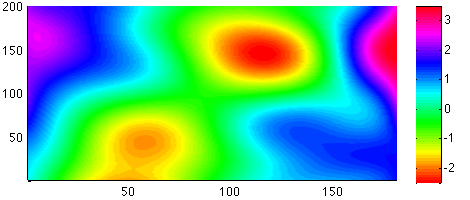}}%
}%
\end{tikzpicture}%
}
\end{center}%
\caption{Panels \subref{sfig-ucobj} and \subref{sfig-ucphase} show reconstructions of object and phase respectively from the data shown in Figure \ref{sfig-lvdata}, which have been obtained by the unconstrained IRGNM. 
The phase is very badly reconstructed which leads to remaining sidelobes in the 
object reconstruction. This is most eminent in the center where the psf features destructive interference. The same deficiencies can be observed
in panels \subref{sfig-projobj} and  \subref{sfig-projphase}, where 
the nonnegativity constraint has been incorporated  by a simple projection
after each step. 
All of the reconstructions depicted in Figures \ref{fig-lv-rec} 
and \ref{fig-lv-recs} were performed with the same regularization parameters.\label{fig-lv-recs}}%
\end{figure}%
To demonstrate the necessity to incorporate nonnegativity constraint into 
the minimization problem, in Figure \ref{fig-lv-recs} we display reconstructions from the same data, without constraint and with simply projecting onto $\mc{C}$
after an unconstrained IRGNM step, i.e.\ for the iteration schemes 
\begin{align*}
(f_{n+1}^{\rm u},\phi_{n+1}^{\rm u}) &
:= \underset{\mathclap{\vphantom{\sum^{F^{2}}}(f,\phi)\in L^{2}(\Omega)\times H^{2}(\Omega)}}{\operatorname{argmin}}
\left[\left\| \FPi'[(f_{n}^{\rm u},\phi_{n}^{\rm u})](f,\phi) -  g_{n}\right \|^2 + \alpha_{n} \|(f,\phi)-(f_0,\phi_0\|^2\right],
\\
(f_{n+1}^{\rm p},\phi_{n+1}^{\rm p}) &:= P_{\mc{C}}\underset{\mathclap{\vphantom{\sum^{F^{2}}}(f,\phi)\in L^{2}(\Omega)\times H^{2}(\Omega)}}{\operatorname{argmin}}\left[\left\| \FPi'[(f_{n}^{\rm p},\phi_{n}^{\rm p})](f,\phi) -  g_{n}\right \|^2 + \alpha_{n} \|(f,\phi)-(f_0,\phi_0)\|^2\right].
\end{align*}
Here the metric projection is given by $P_{\mc{C}}(f,\phi) = (\max(f,0),\phi)$. 
Comparing the reconstructions of Figures \ref{fig-lv-rec} and \ref{fig-lv-recs} it is 
obvious that the incorporation of the nonnegativity constraint in the 
minimization problem is necessary for accurate reconstructions of the phase.

A further option pursued in \cite{vicidomini_etal:10} is to update
$f$ and $\phi$ in alternating manner such that in each update step for $f$
a constrained minimization problem for $f$ only instead of both $f$ and $\phi$
has to be solved. However, such a procedure requires significantly more iteration steps. 

\subsection{Results for simulated and experimental data}
\captionsetup[subfloat]{margin=0ex}%
\begin{figure}[!ht]%
\def\fscb{0.37}
\def\fscc{\fscb*0.308/0.37}
\def\fsca{1} 
\def\optx{-0.2}%
\def\optxr{-\optx}%
\def\optya{-4}%
\def\hspcorr{1}%
\def\xstart{-6.4}%
\def\ystart{-4.76}%
\def\ycorr{\ystart+0.8}%
\def\xstart{-6.55}%
\def\ylength{0.97}%
\begin{center}%
\subfloat[simulated object $f^{\dagger}$\label{sfig-filsimobj}]{%
\begin{tikzpicture}[scale = \fscb]%
\pgftext[]{%
\rotatebox{0}{\includegraphics[scale=\fsca]{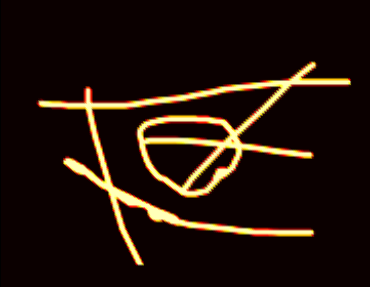}}%
}%
\draw[text=black] (\optxr,\optya) node[right] {{\footnotesize \vphantom{$0$}}};%
\end{tikzpicture}%
}%
\subfloat{%
\begin{tikzpicture}[scale = \fscb]%
\pgftext[]{%
\rotatebox{0}{\includegraphics[scale=\fsca]{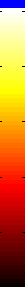}}%
}%
\def\optl{1.54}%
\draw[text=black] (\optxr,\optya) node[right] {{\footnotesize $0$}};%
\draw[text=black] (\optxr,\optya+\optl) node[right] {{\footnotesize $50$}};%
\draw[text=black] (\optxr,\optya+2*\optl) node[right] {{\footnotesize $100$}};%
\draw[text=black] (\optxr,\optya+3*\optl) node[right] {{\footnotesize $150$}};%
\draw[text=black] (\optxr,\optya+4*\optl) node[right] {{\footnotesize $200$}};%
\draw[text=black] (\optxr,\optya+5*\optl) node[right] {{\footnotesize $250$}};%
\end{tikzpicture}%
}%
\addtocounter{subfigure}{-1}%
\subfloat[simulated phase $\phi^{\dagger}$\label{sfig-filsimphase}]{%
\begin{tikzpicture}[scale = \fscb]%
\pgftext[]{%
\rotatebox{0}{\includegraphics[scale=\fsca]{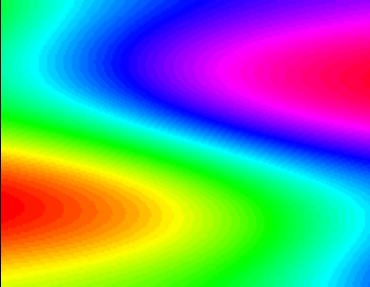}}%
}%
\draw[text=black] (\optxr,\optya) node[right] {{\footnotesize \vphantom{$0$}}};%
\end{tikzpicture}%
}%
\subfloat{%
\begin{tikzpicture}[scale = \fscb]%
\pgftext[]{%
\rotatebox{0}{\includegraphics[scale=\fsca]{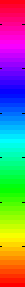}}%
}%
\draw[text=black] (\optxr,\optya) node[right] {{\footnotesize \vphantom{$0$}}};%
\def\optl{1.25}%
\def\optya{-2.9}%
\draw[text=black] (\optxr,\optya) node[right] {{\footnotesize $-2$}};%
\draw[text=black] (\optxr,\optya+\optl) node[right] {{\footnotesize $-1$}};%
\draw[text=black] (\optxr,\optya+2*\optl) node[right] {{\footnotesize $0$}};%
\draw[text=black] (\optxr,\optya+3*\optl) node[right] {{\footnotesize $1$}};%
\draw[text=black] (\optxr,\optya+4*\optl) node[right] {{\footnotesize $2$}};%
\draw[text=black] (\optxr,\optya+5*\optl) node[right] {{\footnotesize $3$}};%
\end{tikzpicture}%
}%
\addtocounter{subfigure}{-1}%
\subfloat[noisy data $g_{\delta}$\label{sfig-filsimdata}]{%
\begin{tikzpicture}[scale = \fscb]%
\pgftext[]{%
\rotatebox{0}{\includegraphics[scale=\fsca]{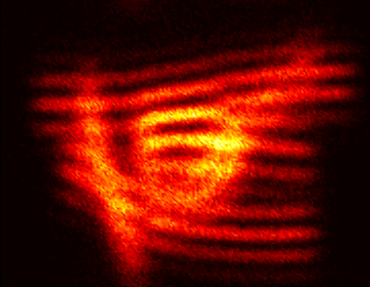}}%
}%
\draw[text=black] (\optxr,\optya) node[right] {{\footnotesize \vphantom{$0$}}};%
\draw[color=white,line width=1pt,->, -latex] plot coordinates{(-4.5,0) (-4.5,-3)};%
\end{tikzpicture}%
}%
\subfloat{%
\def\optl{1.36}%
\begin{tikzpicture}[scale = \fscb]%
\pgftext[]{%
\rotatebox{0}{\includegraphics[scale=\fsca]{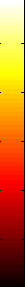}}%
}%
\draw[text=black] (\optxr,\optya) node[right] {{\footnotesize $0$}};%
\draw[text=black] (\optxr,\optya+\optl) node[right] {{\footnotesize $20$}};%
\draw[text=black] (\optxr,\optya+2*\optl) node[right] {{\footnotesize $40$}};%
\draw[text=black] (\optxr,\optya+3*\optl) node[right] {{\footnotesize $60$}};%
\draw[text=black] (\optxr,\optya+4*\optl) node[right] {{\footnotesize $80$}};%
\draw[text=black] (\optxr,\optya+5*\optl) node[right] {{\footnotesize $100$}};%
\end{tikzpicture}%
}\\%
\addtocounter{subfigure}{-1}%
\subfloat[object reconstruction $f_{27}$ for noisy data\label{sfig-filobjrec}]{%
\begin{tikzpicture}[scale = \fscb]%
\pgftext[]{%
\rotatebox{0}{\includegraphics[scale=\fsca]{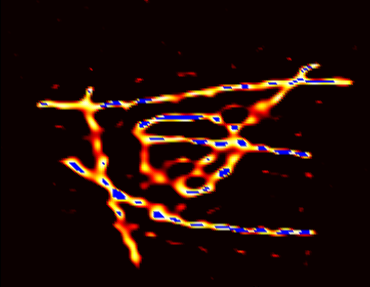}}%
}%
\draw[text=black] (\optxr,\optya) node[right] {{\footnotesize \vphantom{$0$}}};%
\draw[text=black] (0,\ycorr) node[below] {{\footnotesize\vphantom{$0$}}};%
\end{tikzpicture}%
}%
\subfloat{%
\begin{tikzpicture}[scale = \fscb]%
\pgftext[]{%
\rotatebox{0}{\includegraphics[scale=\fsca]{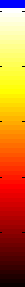}}%
}%
\def\optl{1.54}%
\draw[text=black] (0,\ycorr) node[below] {{\footnotesize\vphantom{$0$}}};%
\draw[text=black] (\optxr,\optya) node[right] {{\footnotesize $0$}};%
\draw[text=black] (\optxr,\optya+\optl) node[right] {{\footnotesize $50$}};%
\draw[text=black] (\optxr,\optya+2*\optl) node[right] {{\footnotesize $100$}};%
\draw[text=black] (\optxr,\optya+3*\optl) node[right] {{\footnotesize $150$}};%
\draw[text=black] (\optxr,\optya+4*\optl) node[right] {{\footnotesize $200$}};%
\draw[text=black] (\optxr,\optya+5*\optl) node[right] {{\footnotesize $250$}};%
\end{tikzpicture}%
}%
\addtocounter{subfigure}{-1}%
\subfloat[phase reconstruction $\phi_{27}$ for noisy data\label{sfig-filphaserec}]{%
\begin{tikzpicture}[scale = \fscb]%
\pgftext[]{%
\rotatebox{0}{\includegraphics[scale=\fsca]{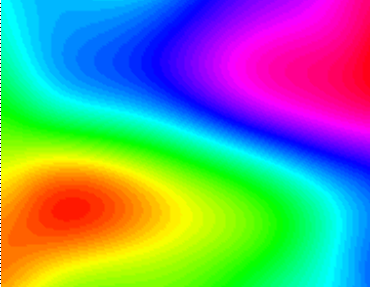}}%
}%
\draw[text=black] (\optxr,\optya) node[right] {{\footnotesize \vphantom{$0$}}};%
\draw[text=black] (0,\ycorr) node[below] {{\footnotesize\vphantom{$0$}}};%
\end{tikzpicture}%
}%
\subfloat{%
\begin{tikzpicture}[scale = \fscb]%
\pgftext[]{%
\rotatebox{0}{\includegraphics[scale=\fsca]{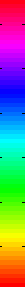}}%
}%
\draw[text=black] (\optxr,\optya) node[right] {{\footnotesize \vphantom{$0$}}};%
\draw[text=black] (0,\ycorr) node[below] {{\footnotesize\vphantom{$0$}}};%
\def\optl{1.25}%
\def\optya{-2.9}%
\draw[text=black] (\optxr,\optya) node[right] {{\footnotesize $-2$}};%
\draw[text=black] (\optxr,\optya+\optl) node[right] {{\footnotesize $-1$}};%
\draw[text=black] (\optxr,\optya+2*\optl) node[right] {{\footnotesize $0$}};%
\draw[text=black] (\optxr,\optya+3*\optl) node[right] {{\footnotesize $1$}};%
\draw[text=black] (\optxr,\optya+4*\optl) node[right] {{\footnotesize $2$}};%
\draw[text=black] (\optxr,\optya+5*\optl) node[right] {{\footnotesize $3$}};%
\end{tikzpicture}%
}%
\addtocounter{subfigure}{-1}%
\subfloat[errors and residual for noisy data\label{sfig-filresid}]{%
\begin{tikzpicture}[scale = \fscc]%
\pgftext[]{%
\rotatebox{0}{\includegraphics[scale=\fsca]{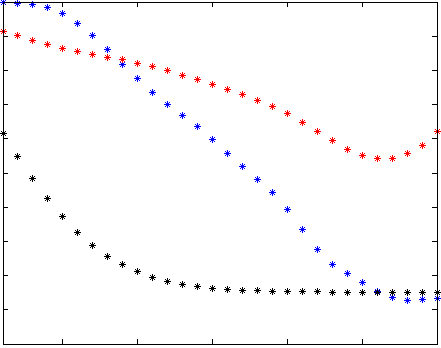}}%
}%
\def\oxr{0.3}%
\def\oya{4.3}%
\def\oline{0.6}%
\def\parone{1}%
\def\partwo{1.8}%
\def\xlength{2.09}%
\draw[text=black] (\optxr,\optya-0.8) node[right] {{\footnotesize \vphantom{$0$}}};%
\draw[text=black] (\oxr,\oya) node[right] {\parbox{\parone cm}{{\tiny residual}}};%
\draw[text=black] (\oxr-1,\oya) node[right] {\parbox{\parone cm}{{\tiny $*$}}};%
\draw[text=red] (\oxr,\oya-\oline) node[right] {\parbox{\partwo cm}{{\tiny object error}}};%
\draw[text=red] (\oxr-1,\oya-\oline) node[right] {\parbox{\partwo cm}{{\tiny $*$}}};%
\draw[text=blue] (\oxr,\oya-2*\oline) node[right] {\parbox{\partwo cm}{{\tiny phase error}}};%
\draw[text=blue] (\oxr-1,\oya-2*\oline) node[right] {\parbox{\partwo cm}{{\tiny $*$}}};%
\draw[text=black] (\xstart,\ystart) node[left] {{\footnotesize$0$}};%
\draw[text=black] (\xstart,\ystart+\ylength) node[left] {{\footnotesize$0.1$}};%
\draw[text=black] (\xstart,\ystart+2*\ylength) node[left] {{\footnotesize$0.2$}};%
\draw[text=black] (\xstart,\ystart+3*\ylength) node[left] {{\footnotesize$0.3$}};%
\draw[text=black] (\xstart,\ystart+4*\ylength) node[left] {{\footnotesize$0.4$}};%
\draw[text=black] (\xstart,\ystart+5*\ylength) node[left] {{\footnotesize$0.5$}};%
\draw[text=black] (\xstart,\ystart+6*\ylength) node[left] {{\footnotesize$0.6$}};%
\draw[text=black] (\xstart,\ystart+7*\ylength) node[left] {{\footnotesize$0.7$}};%
\draw[text=black] (\xstart,\ystart+8*\ylength) node[left] {{\footnotesize$0.8$}};%
\draw[text=black] (\xstart,\ystart+9*\ylength) node[left] {{\footnotesize$0.9$}};%
\draw[text=black] (\xstart,\ystart+10*\ylength) node[left] {{\footnotesize$1$}};%
\draw[text=black] (\xstart+\xlength,\ystart) node[below] {{\footnotesize$5$}};%
\draw[text=black] (\xstart+2*\xlength,\ystart) node[below] {{\footnotesize$10$}};%
\draw[text=black] (\xstart+3*\xlength,\ystart) node[below] {{\footnotesize$15$}};%
\draw[text=black] (\xstart+4*\xlength,\ystart) node[below] {{\footnotesize$20$}};%
\draw[text=black] (\xstart+5*\xlength,\ystart) node[below] {{\footnotesize$25$}};%
\draw[text=black] (\xstart+6*\xlength,\ystart) node[below] {{\footnotesize$30$}};%
\end{tikzpicture}%
}\\%
\subfloat[object reconstruction $f_{36}$ for exact data\label{sfig-filexobjrec}]{%
\begin{tikzpicture}[scale = \fscb]%
\pgftext[]{%
\rotatebox{0}{\includegraphics[scale=\fsca]{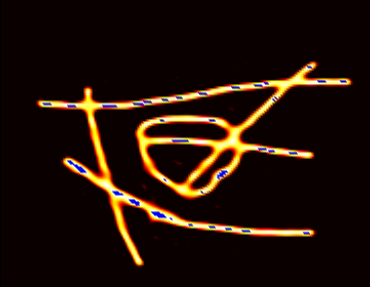}}%
}%
\draw[text=black] (\optxr,\optya) node[right] {{\footnotesize \vphantom{$0$}}};%
\draw[text=black] (0,\ycorr) node[below] {{\footnotesize\vphantom{$0$}}};%
\end{tikzpicture}%
}%
\subfloat{%
\begin{tikzpicture}[scale = \fscb]%
\pgftext[]{%
\rotatebox{0}{\includegraphics[scale=\fsca]{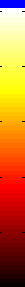}}%
}%
\def\optl{1.54}%
\draw[text=black] (0,\ycorr) node[below] {{\footnotesize\vphantom{$0$}}};%
\draw[text=black] (\optxr,\optya) node[right] {{\footnotesize $0$}};%
\draw[text=black] (\optxr,\optya+\optl) node[right] {{\footnotesize $50$}};%
\draw[text=black] (\optxr,\optya+2*\optl) node[right] {{\footnotesize $100$}};%
\draw[text=black] (\optxr,\optya+3*\optl) node[right] {{\footnotesize $150$}};%
\draw[text=black] (\optxr,\optya+4*\optl) node[right] {{\footnotesize $200$}};%
\draw[text=black] (\optxr,\optya+5*\optl) node[right] {{\footnotesize $250$}};%
\end{tikzpicture}%
}%
\addtocounter{subfigure}{-1}%
\subfloat[phase reconstruction $\phi_{36}$ for exact data\label{sfig-filexphaserec}]{%
\begin{tikzpicture}[scale = \fscb]%
\pgftext[]{%
\rotatebox{0}{\includegraphics[scale=\fsca]{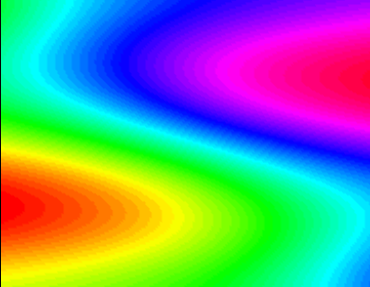}}%
}%
\draw[text=black] (\optxr,\optya) node[right] {{\footnotesize \vphantom{$0$}}};%
\draw[text=black] (0,\ycorr) node[below] {{\footnotesize\vphantom{$0$}}};%
\end{tikzpicture}%
}%
\subfloat{%
\begin{tikzpicture}[scale = \fscb]%
\pgftext[]{%
\rotatebox{0}{\includegraphics[scale=\fsca]{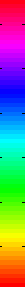}}%
}%
\draw[text=black] (\optxr,\optya) node[right] {{\footnotesize \vphantom{$0$}}};%
\draw[text=black] (0,\ycorr) node[below] {{\footnotesize\vphantom{$0$}}};%
\def\optl{1.25}%
\def\optya{-2.9}%
\draw[text=black] (\optxr,\optya) node[right] {{\footnotesize $-2$}};%
\draw[text=black] (\optxr,\optya+\optl) node[right] {{\footnotesize $-1$}};%
\draw[text=black] (\optxr,\optya+2*\optl) node[right] {{\footnotesize $0$}};%
\draw[text=black] (\optxr,\optya+3*\optl) node[right] {{\footnotesize $1$}};%
\draw[text=black] (\optxr,\optya+4*\optl) node[right] {{\footnotesize $2$}};%
\draw[text=black] (\optxr,\optya+5*\optl) node[right] {{\footnotesize $3$}};%
\end{tikzpicture}%
}%
\addtocounter{subfigure}{-1}%
\subfloat[errors and residual for exact data\label{sfig-ex-resid}]{%
\begin{tikzpicture}[scale = \fscc]%
\pgftext[]{%
\rotatebox{0}{\includegraphics[scale=\fsca]{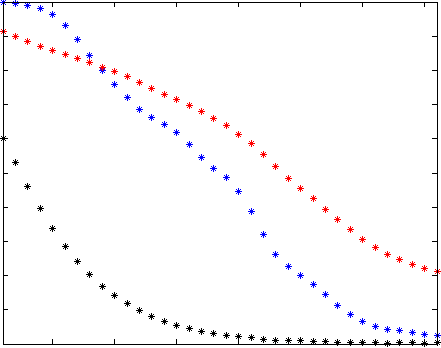}}%
}%
\def\oxr{0.3}%
\def\oya{4.3}%
\def\oline{0.6}%
\def\parone{1}%
\def\partwo{1.8}%
\def\xlength{1.73}%
\draw[text=black] (\optxr,\optya-0.8) node[right] {{\footnotesize \vphantom{$0$}}};%
\draw[text=black] (\oxr,\oya) node[right] {\parbox{\parone cm}{{\tiny residual}}};%
\draw[text=black] (\oxr-1,\oya) node[right] {\parbox{\parone cm}{{\tiny $*$}}};%
\draw[text=red] (\oxr,\oya-\oline) node[right] {\parbox{\partwo cm}{{\tiny object error}}};%
\draw[text=red] (\oxr-1,\oya-\oline) node[right] {\parbox{\partwo cm}{{\tiny $*$}}};%
\draw[text=blue] (\oxr,\oya-2*\oline) node[right] {\parbox{\partwo cm}{{\tiny phase error}}};%
\draw[text=blue] (\oxr-1,\oya-2*\oline) node[right] {\parbox{\partwo cm}{{\tiny $*$}}};%
\draw[text=black] (\xstart,\ystart) node[left] {{\footnotesize$0$}};%
\draw[text=black] (\xstart,\ystart+\ylength) node[left] {{\footnotesize$0.1$}};%
\draw[text=black] (\xstart,\ystart+2*\ylength) node[left] {{\footnotesize$0.2$}};%
\draw[text=black] (\xstart,\ystart+3*\ylength) node[left] {{\footnotesize$0.3$}};%
\draw[text=black] (\xstart,\ystart+4*\ylength) node[left] {{\footnotesize$0.4$}};%
\draw[text=black] (\xstart,\ystart+5*\ylength) node[left] {{\footnotesize$0.5$}};%
\draw[text=black] (\xstart,\ystart+6*\ylength) node[left] {{\footnotesize$0.6$}};%
\draw[text=black] (\xstart,\ystart+7*\ylength) node[left] {{\footnotesize$0.7$}};%
\draw[text=black] (\xstart,\ystart+8*\ylength) node[left] {{\footnotesize$0.8$}};%
\draw[text=black] (\xstart,\ystart+9*\ylength) node[left] {{\footnotesize$0.9$}};%
\draw[text=black] (\xstart,\ystart+10*\ylength) node[left] {{\footnotesize$1$}};%
\def\xcorr{-0.1}%
\draw[text=black] (\xstart-\xcorr+\xlength,\ystart) node[below] {{\footnotesize$5$}};%
\draw[text=black] (\xstart-\xcorr+2*\xlength,\ystart) node[below] {{\footnotesize$10$}};%
\draw[text=black] (\xstart-\xcorr+3*\xlength,\ystart) node[below] {{\footnotesize$15$}};%
\draw[text=black] (\xstart-\xcorr+4*\xlength,\ystart) node[below] {{\footnotesize$20$}};%
\draw[text=black] (\xstart-\xcorr+5*\xlength,\ystart) node[below] {{\footnotesize$25$}};%
\draw[text=black] (\xstart-\xcorr+6*\xlength,\ystart) node[below] {{\footnotesize$30$}};%
\draw[text=black] (\xstart-\xcorr+7*\xlength,\ystart) node[below] {{\footnotesize$35$}};%
\end{tikzpicture}%
}%
\end{center}%
\caption{Panels \subref{sfig-filsimobj} and \subref{sfig-filsimphase} show a simulated object and phase, and panel \subref{sfig-filsimdata} the corresponding 4Pi-data perturbed with Poisson-noise. Panels \subref{sfig-filobjrec} and \subref{sfig-filphaserec} show the reconstructions of object and phase respectively from this noisy data obtained with the constrained IRGNM. In panel \subref{sfig-filresid} the residual $\|g_{\delta}-F(f_n,\phi_n)\|_{L^2}$, the object error
$\|f_n-f^{\dagger}\|_{L^2}$, and the phase error $\|\phi_n-\phi^{\dagger}\|_{L^2}$ 
are plotted over the iteration index $n$. 
Panels \subref{sfig-filexobjrec}--\subref{sfig-ex-resid} are
analogous to panels \subref{sfig-filobjrec}--\subref{sfig-filresid} 
with noisy data $g_{\delta}$ replaced by exact data $F(f^{\dagger},\phi^{\dagger})$. 
\label{fig-filsimdata}}%
\end{figure}%
\captionsetup[subfloat]{margin=0ex}
Figure \ref{fig-filsimdata} shows reconstructions from simulated 
two-dimensional noisy and exact data. Here we chose polynomials of maximal 
degree 7 in each dimension to approximate the reconstructed phase. 
We chose the exact phase as a shifted sum of a sine and arctan function, 
which does not belong to the polynomial subspace. 
For exact data the sidelobes are removed completely, and for noisy data at
the given count rate only very little of the sidelobes is left in the
reconstruction. The required number of 
semi-smooth Newton (SSN) steps increases with $n$. To give an idea, 
we mention that less than 8 SSN steps were needed
for $n\leq 21$ with less than 80 CG steps in each SSN step,
and for $n=30$ the algorithm required 49 SSN steps with less than 
600 CG steps. We must say that the stopping indices for the 
Gau{\ss}-Newton iteration are chosen somewhat arbitrarily in this paper. 
The development of a good stopping rule for the kind of errors considered
in this paper, nonlinear operators and convex constraints is an interesting 
topic for future research. 

\captionsetup[subfloat]{margin=0ex}%
\begin{figure}[!ht]%
\def\fscb{0.37}
\def\fscc{\fscb*0.308/0.37}
\def\fsca{1} 
\def\optx{-0.2}%
\def\optxr{-\optx}%
\def\optya{-4.25}%
\def\hspcorr{1}%
\def\xstart{-6}%
\def\ystart{-4.76}%
\def\ycorr{\ystart+0.8}%
\def\ylength{1.06}%
\def\xstart{-6.2}%
\def\xlength{2.05}%
\begin{center}%
\subfloat[simulated object $f^{\dagger}$\label{sfig-constobject}]{%
\begin{tikzpicture}[scale = \fscb]%
\pgftext[]{%
\rotatebox{0}{\includegraphics[scale=\fsca]{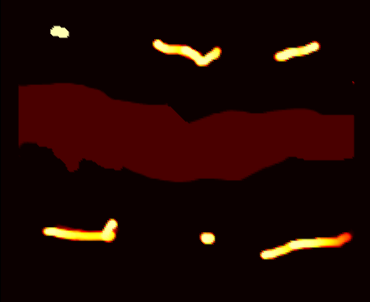}}%
}%
\draw[text=black] (\optxr,\optya) node[right] {{\footnotesize \vphantom{$0$}}};%
\end{tikzpicture}%
}%
\subfloat{%
\begin{tikzpicture}[scale = \fscb]%
\pgftext[]{%
\rotatebox{0}{\includegraphics[scale=\fsca]{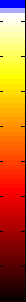}}%
}%
\def\optl{0.99}%
\draw[text=black] (\optxr,\optya) node[right] {{\footnotesize $0$}};%
\draw[text=black] (\optxr,\optya+\optl) node[right] {{\footnotesize $50$}};%
\draw[text=black] (\optxr,\optya+2*\optl) node[right] {{\footnotesize $100$}};%
\draw[text=black] (\optxr,\optya+3*\optl) node[right] {{\footnotesize $150$}};%
\draw[text=black] (\optxr,\optya+4*\optl) node[right] {{\footnotesize $200$}};%
\draw[text=black] (\optxr,\optya+5*\optl) node[right] {{\footnotesize $250$}};%
\draw[text=black] (\optxr,\optya+6*\optl) node[right] {{\footnotesize $300$}};%
\draw[text=black] (\optxr,\optya+7*\optl) node[right] {{\footnotesize $350$}};%
\draw[text=black] (\optxr,\optya+8*\optl) node[right] {{\footnotesize $400$}};%
\end{tikzpicture}%
}%
\addtocounter{subfigure}{-1}%
\subfloat[simulated phase $\phi^{\dagger}$\label{sfig-constphiex}]{%
\begin{tikzpicture}[scale = \fscb]%
\pgftext[]{%
\rotatebox{0}{\includegraphics[scale=\fsca]{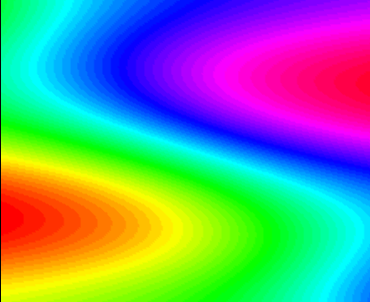}}%
}%
\draw[text=black] (\optxr,\optya) node[right] {{\footnotesize \vphantom{$0$}}};%
\end{tikzpicture}%
}%
\subfloat{%
\begin{tikzpicture}[scale = \fscb]%
\pgftext[]{%
\rotatebox{0}{\includegraphics[scale=\fsca]{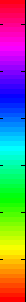}}%
}%
\draw[text=black] (\optxr,\optya) node[right] {{\footnotesize \vphantom{$0$}}};%
\def\optl{1.345}%
\def\optya{-3.1}%
\draw[text=black] (\optxr,\optya) node[right] {{\footnotesize $-2$}};%
\draw[text=black] (\optxr,\optya+\optl) node[right] {{\footnotesize $-1$}};%
\draw[text=black] (\optxr,\optya+2*\optl) node[right] {{\footnotesize $0$}};%
\draw[text=black] (\optxr,\optya+3*\optl) node[right] {{\footnotesize $1$}};%
\draw[text=black] (\optxr,\optya+4*\optl) node[right] {{\footnotesize $2$}};%
\draw[text=black] (\optxr,\optya+5*\optl) node[right] {{\footnotesize $3$}};%
\end{tikzpicture}%
}%
\addtocounter{subfigure}{-1}%
\subfloat[noisy data $g_{\delta}$ \label{sfig-constdata}]{%
\begin{tikzpicture}[scale = \fscb]%
\pgftext[]{%
\rotatebox{0}{\includegraphics[scale=\fsca]{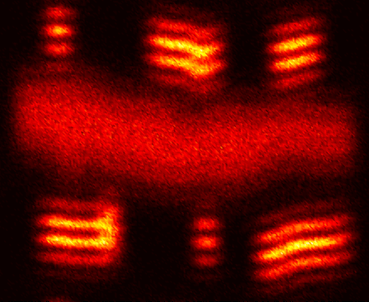}}%
}%
\draw[text=black] (\optxr,\optya) node[right] {{\footnotesize \vphantom{$0$}}};%
\draw[color=white,line width=1pt,->, -latex] plot coordinates{(-4.8,-0.5) (-4.8,-3.5)};%
\end{tikzpicture}%
}%
\subfloat{%
\def\optl{1.34}%
\begin{tikzpicture}[scale = \fscb]%
\pgftext[]{%
\rotatebox{0}{\includegraphics[scale=\fsca]{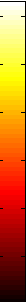}}%
}%
\draw[text=black] (\optxr,\optya) node[right] {{\footnotesize $0$}};%
\draw[text=black] (\optxr,\optya+\optl) node[right] {{\footnotesize $20$}};%
\draw[text=black] (\optxr,\optya+2*\optl) node[right] {{\footnotesize $40$}};%
\draw[text=black] (\optxr,\optya+3*\optl) node[right] {{\footnotesize $60$}};%
\draw[text=black] (\optxr,\optya+4*\optl) node[right] {{\footnotesize $80$}};%
\draw[text=black] (\optxr,\optya+5*\optl) node[right] {{\footnotesize $100$}};%
\draw[text=black] (\optxr,\optya+6*\optl) node[right] {{\footnotesize $120$}};%
\end{tikzpicture}%
}\\%
\addtocounter{subfigure}{-1}%
\subfloat[object reconstruction $f_{14}$\label{sfig-constobjrec}]{%
\begin{tikzpicture}[scale = \fscb]%
\pgftext[]{%
\rotatebox{0}{\includegraphics[scale=\fsca]{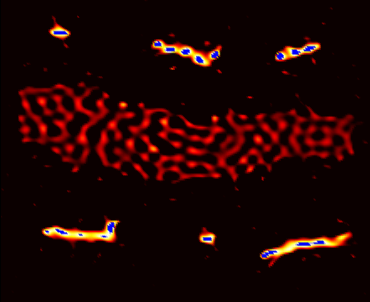}}%
}%
\draw[text=black] (\optxr,\optya) node[right] {{\footnotesize \vphantom{$0$}}};%
\end{tikzpicture}%
}%
\subfloat{%
\begin{tikzpicture}[scale = \fscb]%
\pgftext[]{%
\rotatebox{0}{\includegraphics[scale=\fsca]{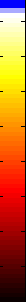}}%
}%
\def\optl{0.99}%
\draw[text=black] (\optxr,\optya) node[right] {{\footnotesize $0$}};%
\draw[text=black] (\optxr,\optya+\optl) node[right] {{\footnotesize $50$}};%
\draw[text=black] (\optxr,\optya+2*\optl) node[right] {{\footnotesize $100$}};%
\draw[text=black] (\optxr,\optya+3*\optl) node[right] {{\footnotesize $150$}};%
\draw[text=black] (\optxr,\optya+4*\optl) node[right] {{\footnotesize $200$}};%
\draw[text=black] (\optxr,\optya+5*\optl) node[right] {{\footnotesize $250$}};%
\draw[text=black] (\optxr,\optya+6*\optl) node[right] {{\footnotesize $300$}};%
\draw[text=black] (\optxr,\optya+7*\optl) node[right] {{\footnotesize $350$}};%
\draw[text=black] (\optxr,\optya+8*\optl) node[right] {{\footnotesize $400$}};%
\end{tikzpicture}%
}%
\addtocounter{subfigure}{-1}%
\subfloat[phase reconstruction $\phi_{14}$\label{sfig-constphaserec}]{%
\begin{tikzpicture}[scale = \fscb]%
\pgftext[]{%
\rotatebox{0}{\includegraphics[scale=\fsca]{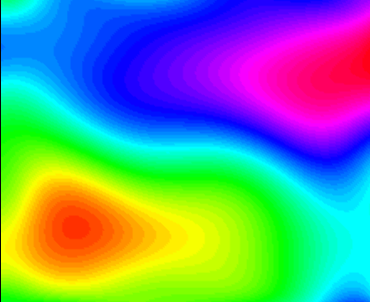}}%
}%
\draw[text=black] (\optxr,\optya) node[right] {{\footnotesize \vphantom{$0$}}};%
\end{tikzpicture}%
}%
\subfloat{%
\begin{tikzpicture}[scale = \fscb]%
\pgftext[]{%
\rotatebox{0}{\includegraphics[scale=\fsca]{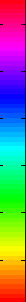}}%
}%
\draw[text=black] (\optxr,\optya) node[right] {{\footnotesize \vphantom{$0$}}};%
\def\optl{1.345}%
\def\optya{-3.1}%
\draw[text=black] (\optxr,\optya) node[right] {{\footnotesize $-2$}};%
\draw[text=black] (\optxr,\optya+\optl) node[right] {{\footnotesize $-1$}};%
\draw[text=black] (\optxr,\optya+2*\optl) node[right] {{\footnotesize $0$}};%
\draw[text=black] (\optxr,\optya+3*\optl) node[right] {{\footnotesize $1$}};%
\draw[text=black] (\optxr,\optya+4*\optl) node[right] {{\footnotesize $2$}};%
\draw[text=black] (\optxr,\optya+5*\optl) node[right] {{\footnotesize $3$}};%
\end{tikzpicture}%
}%
\addtocounter{subfigure}{-1}%
\subfloat[EM-TV object reconstruction\label{sfig-consttvrec}]{%
\begin{tikzpicture}[scale = \fscb]%
\pgftext[]{%
\rotatebox{0}{\includegraphics[scale=\fsca]{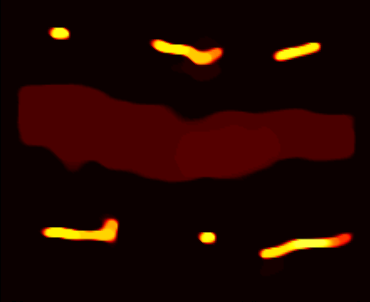}}%
}%
\draw[text=black] (\optxr,\optya) node[right] {{\footnotesize \vphantom{$0$}}};%
\end{tikzpicture}%
}%
\subfloat{%
\begin{tikzpicture}[scale = \fscb]%
\pgftext[]{%
\rotatebox{0}{\includegraphics[scale=\fsca]{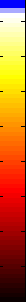}}%
}%
\def\optl{0.99}%
\draw[text=black] (\optxr,\optya) node[right] {{\footnotesize $0$}};%
\draw[text=black] (\optxr,\optya+\optl) node[right] {{\footnotesize $50$}};%
\draw[text=black] (\optxr,\optya+2*\optl) node[right] {{\footnotesize $100$}};%
\draw[text=black] (\optxr,\optya+3*\optl) node[right] {{\footnotesize $150$}};%
\draw[text=black] (\optxr,\optya+4*\optl) node[right] {{\footnotesize $200$}};%
\draw[text=black] (\optxr,\optya+5*\optl) node[right] {{\footnotesize $250$}};%
\draw[text=black] (\optxr,\optya+6*\optl) node[right] {{\footnotesize $300$}};%
\draw[text=black] (\optxr,\optya+7*\optl) node[right] {{\footnotesize $350$}};%
\draw[text=black] (\optxr,\optya+8*\optl) node[right] {{\footnotesize $400$}};%
\end{tikzpicture}%
}%
\addtocounter{subfigure}{-1}%
\end{center}%
\caption{Panels \subref{sfig-constobject} and \subref{sfig-constphiex} show a simulated object and phase from which noisy 4Pi data was created (panel \subref{sfig-constdata}). Panels \subref{sfig-constobjrec} depicts the object reconstruction obtained with the IRGNM and panel \subref{sfig-constphaserec} the corresponding phase reconstruction.  Panel \subref{sfig-consttvrec} shows a reconstruction by an expectation maximization algorithm with TV penalty using the 
reconstructed phase \subref{sfig-constphaserec}.%
 \label{fig-const}}%
\end{figure}%
\captionsetup[subfloat]{margin=0ex}

In Figure \ref{fig-const} we chose an object which is constant in a region, 
and hence the data carry no information on the phase there. 
Due to the $H^{2}$-phase penalty term, the phase is interpolated smoothly
in this area and recovered quite well, except in dark areas close to the 
boundary. In contrast, the reconstruction of the object exhibits a grainy 
structure in the central area. This is a consequence of choosing the $L^2$ 
norm as object penalty. Since we have found a good approximation 
$\phi_{\rm app}$ of the phase, we can compute a better reconstruction of the 
object in a second step by solving an inverse problem
for the linear operator $f\mapsto F(f,\phi_{\rm app})$.   
The result in Figure \ref{sfig-constphaserec} was computed
using an expectation-maximization method with a TV penalty term
and Bregman iterations as described in \cite{BSB09}.

Figure \ref{fig-3d-dat} shows cuts through 3-dimensional experimental data. 
The corresponding reconstructions of object and phase are shown in Figure \ref{fig-3d-recs}. 
Note that due to the simultaneous reconstruction of the phase function the 
non-symmetric sidelobes in the data have been removed in the object reconstruction. %
\def\fsca{0.8}
\def\foaxs{2.5}%
\def\foays{-5.5}%
\def\foaxe{4.6}%
\def\foaye{-4.6}%
\def\fayx{-2}%
\def\fayy{-1.8}%
\def\fazx{-6.5}%
\def\fazy{2.5}%
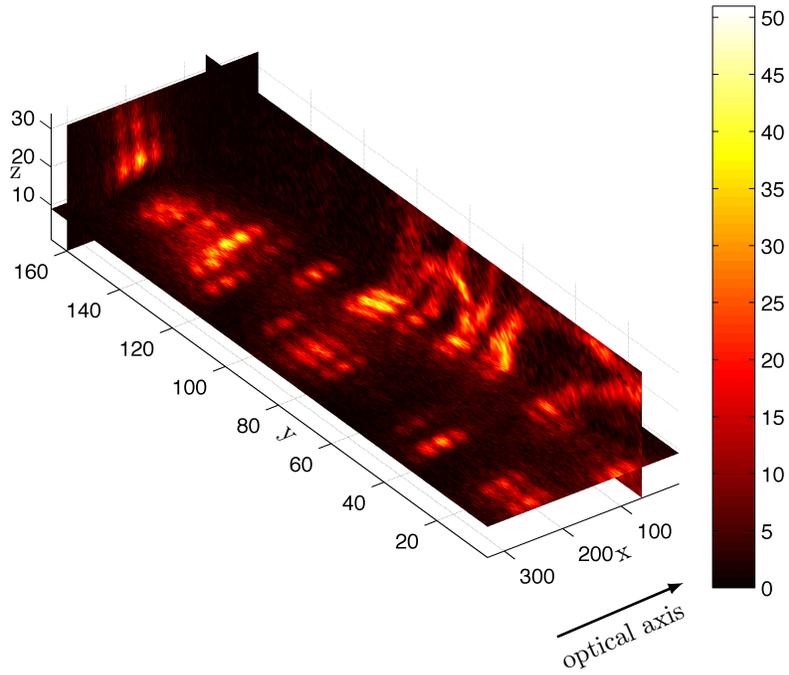
\begin{figure}[htbp]%
\def\fsccb{1.16}
\def\fscainv{1/0.412}%
\def\tikzhc{-1.8}%
\def\xopt{-0.3}%
\def\optx{\xopt}
\def\optya{-1.7}%
\def\fcbcor{0.06}%
\def\fsfsp{3}%
\begin{center}%
\begin{tikzpicture}[scale = \fsca]%
\pgftext[]{%
\rotatebox{0}{\includegraphics[scale=1]{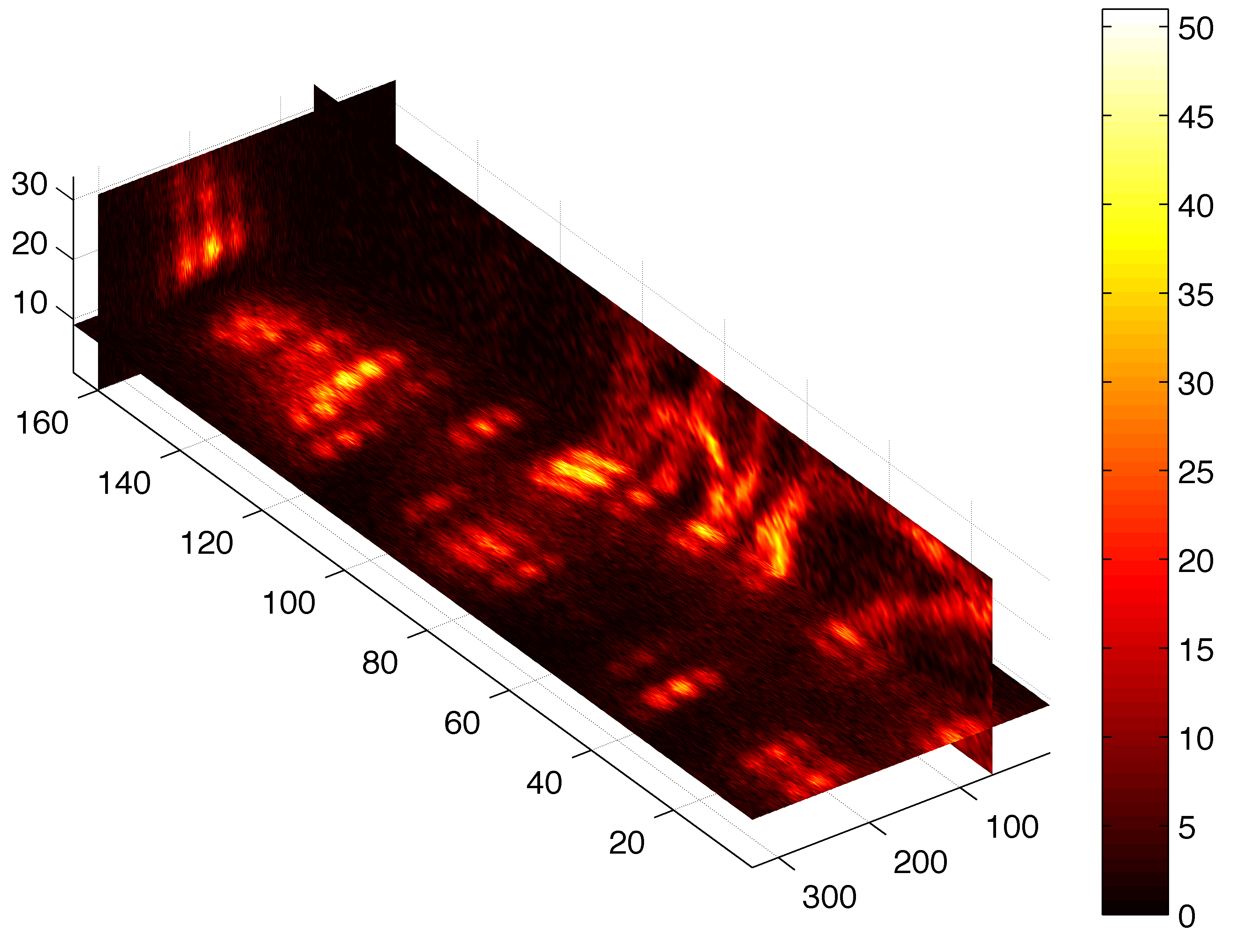}}%
}
\draw[color=black,line width=1pt,->,-latex] plot coordinates{(\foaxs,\foays) (\foaxe,\foaye)};%
\draw[text=black] (\foaxs+1.1,\foays+0.8) node[below] {\rotatebox{24}{\footnotesize optical axis}};%
\draw[text=black] (\foaxs+1.1,\foaye+0.8) node[below] {\rotatebox{24}{\footnotesize x}};%
\draw[text=black] (\fayx,\fayy) node[below] {\rotatebox{-44}{\footnotesize y}};%
\draw[text=black] (\fazx,\fazy) node[below] {\rotatebox{0}{\footnotesize z}};%
\end{tikzpicture}%
\end{center}%
\caption{Data of microtubules in a Vero cell, for ${\rm NA}=1.34$, $\lambda_{\rm ex}=635{\rm nm}$, $\lambda_{\rm em}= 680 {\rm nm}$. The data extension is $(2952 {\rm nm} \times 9296 {\rm nm}\times 1904 {\rm nm})$ in $x,y$ and $z$ direction respectively. The annotations at the axes number the voxels in the respective dimension.}\label{fig-3d-dat}%
\end{figure}
Including the zero padding, the data contained approximately $2$ million voxels. The phase has been approximated by Chebychev polynomials of ma\-xi\-mal degree $3$ in each dimension. %
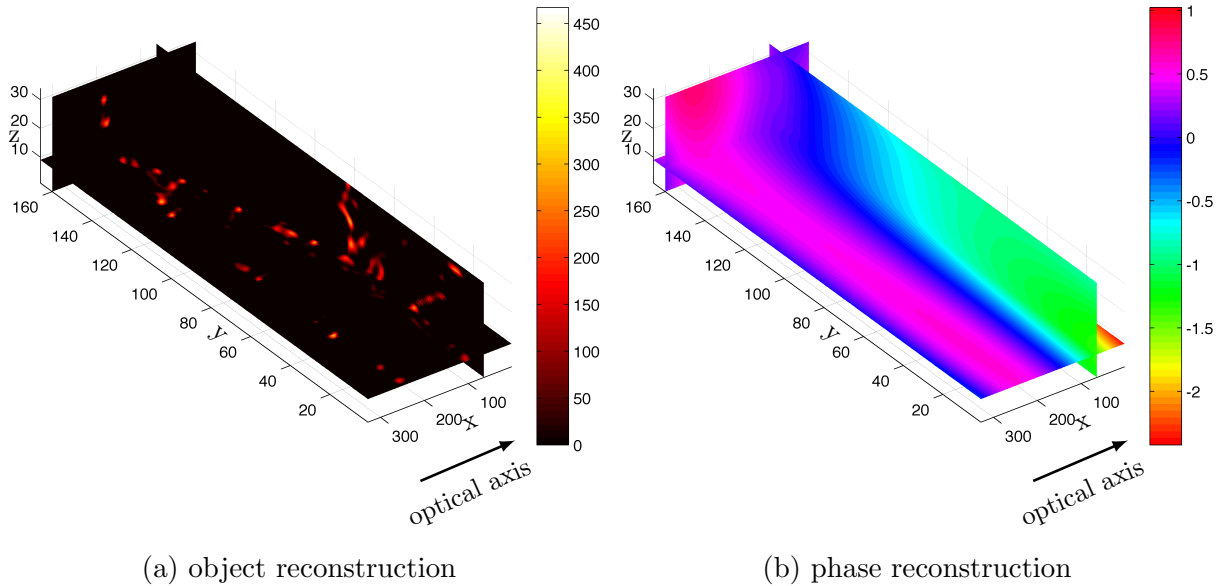
\begin{figure}[htbp]%
\def\fsca{0.6}
\def\fsccb{1.16}
\def\fscainv{1/0.412}%
\def\tikzhc{-1.8}%
\def\xopt{-0.3}%
\def\optx{\xopt}
\def\optya{-1.7}%
\def\fcbcor{0.06}%
\def\fsfsp{3}%
\begin{center}%
\subfloat[object reconstruction\label{sfig-3d-obj}]{
\begin{tikzpicture}[scale = \fsca]%
\pgftext[]{%
\rotatebox{0}{\includegraphics[scale=1]{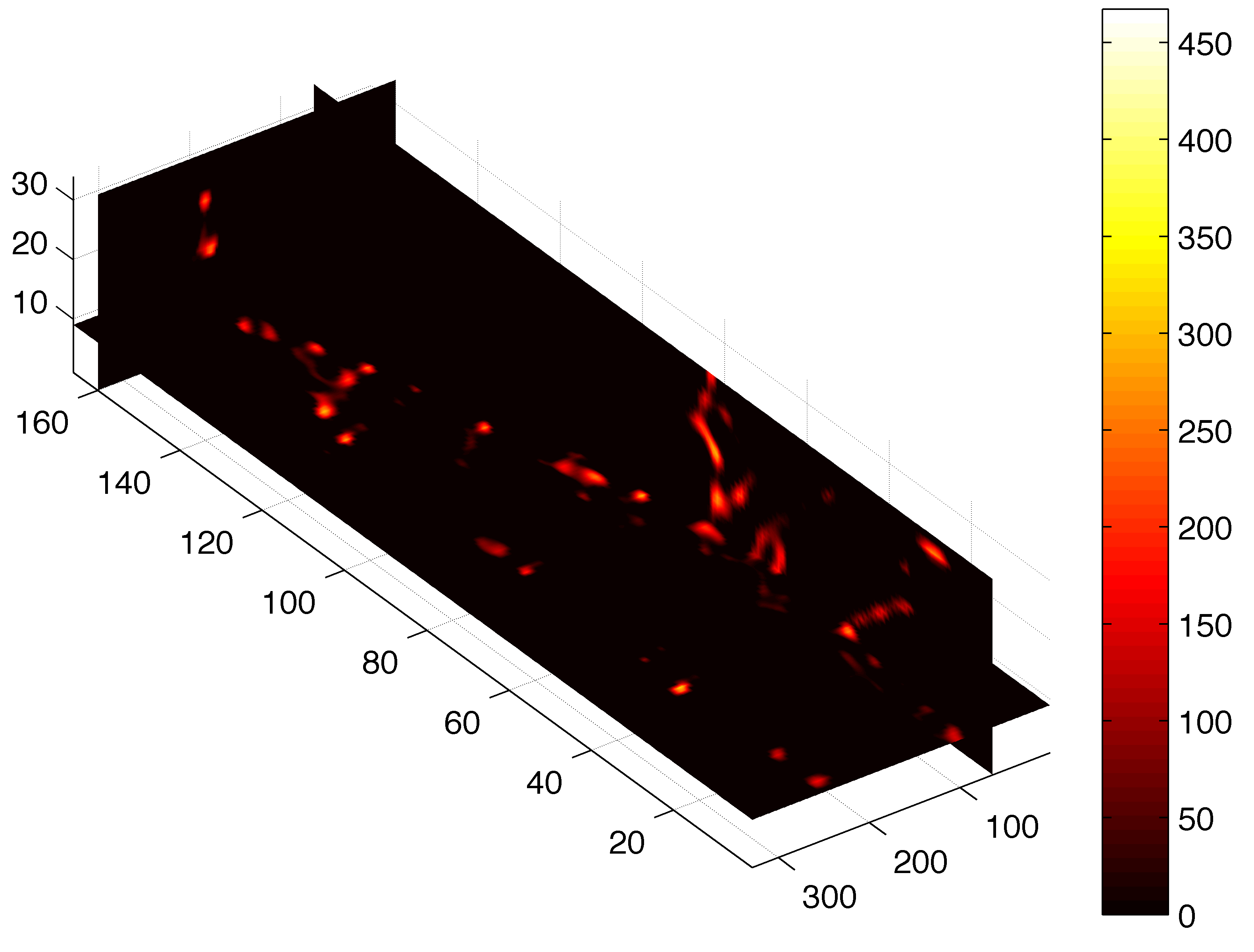}}%
}
\draw[color=black,line width=1pt,->,-latex] plot coordinates{(\foaxs,\foays) (\foaxe,\foaye)};%
\draw[text=black] (\foaxs+1.1,\foays+0.8) node[below] {\rotatebox{24}{\footnotesize optical axis}};%
\draw[text=black] (\foaxs+1.1,\foaye+0.8) node[below] {\rotatebox{24}{\footnotesize x}};%
\draw[text=black] (\fayx,\fayy) node[below] {\rotatebox{-44}{\footnotesize y}};%
\draw[text=black] (\fazx,\fazy) node[below] {\rotatebox{0}{\footnotesize z}};
\end{tikzpicture}%
}%
%
%
%
\subfloat[phase reconstruction\label{sfig-3d-phase}]{%
\begin{tikzpicture}[scale = \fsca]%
\pgftext[]{%
\rotatebox{0}{\includegraphics[scale=1]{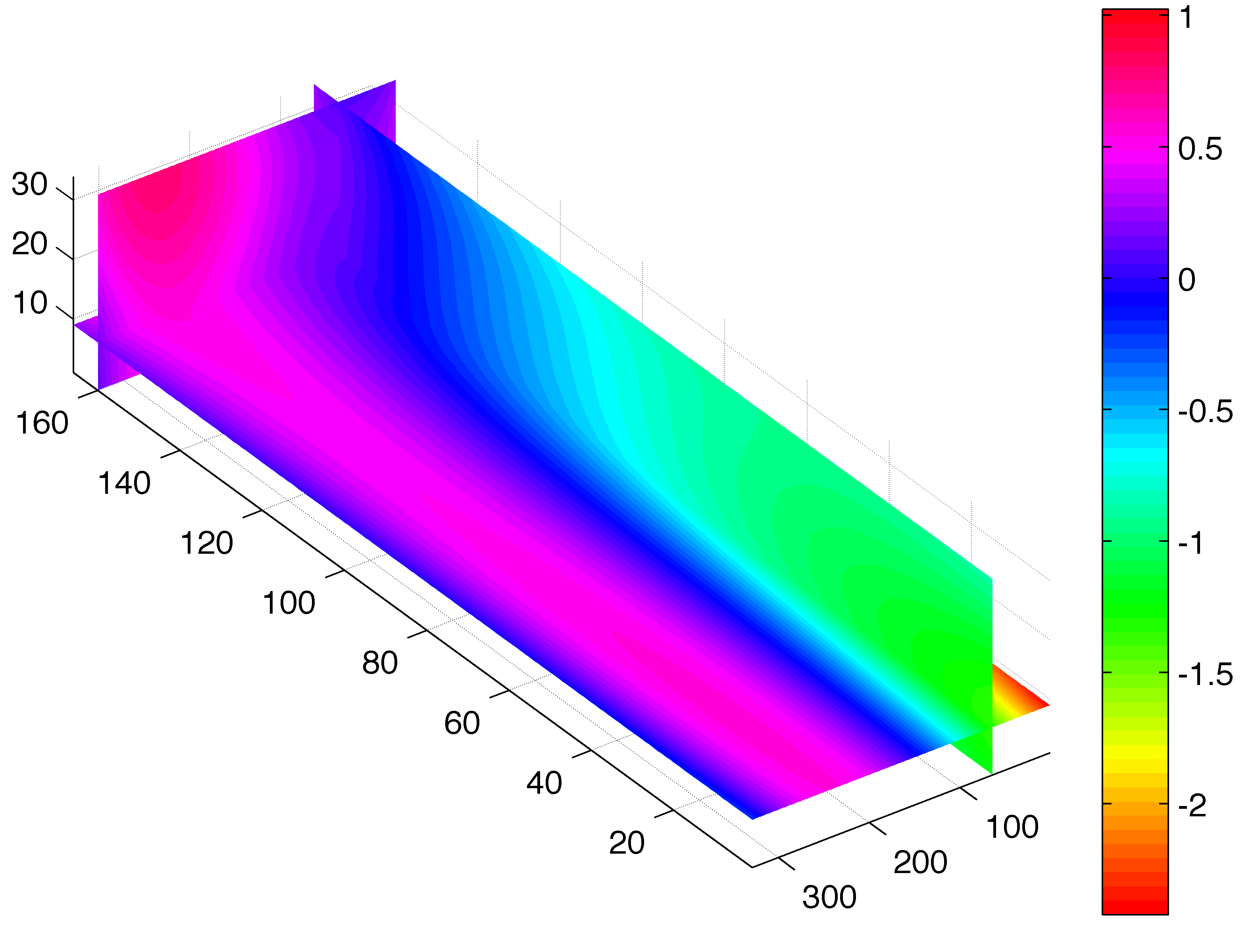}}%
}
\draw[color=black,line width=1pt,->,-latex] plot coordinates{(\foaxs,\foays) (\foaxe,\foaye)};%
\draw[text=black] (\foaxs+1.1,\foays+0.8) node[below] {\rotatebox{24}{\footnotesize optical axis}};%
\draw[text=black] (\foaxs+1.1,\foaye+0.8) node[below] {\rotatebox{24}{\footnotesize x}};%
\draw[text=black] (\fayx,\fayy) node[below] {\rotatebox{-44}{\footnotesize y}};%
\draw[text=black] (\fazx,\fazy) node[below] {\rotatebox{0}{\footnotesize z}};%
\end{tikzpicture}%
}%
\end{center}%
\caption{Panel \subref{sfig-3d-obj} shows the reconstruction of the object from the data shown in Figure \ref{fig-3d-dat}. In Panel \subref{sfig-3d-phase} the corresponding reconstruction of the phase is depicted.\label{fig-3d-recs}}%
\end{figure}

\FloatBarrier
\section*{Acknowledgments}
We would like thank Andreas Sch\"onle and Giuseppe Vicidomini (MPI biophysical chemistry, G\"ottingen) for their support concerning microscopic data and 
information on 4Pi-microscopy. We further thank Jan Keller for his matlab code 
implementing the 4Pi-psf model \eqref{eq:accurate4PiPSF}. 
Finally, we gratefully acknowledge financial support by BMBF 
(German Federal Ministry for Education and Science) under the project INVERS.
\bibliography{paper}
\bibliographystyle{abbrv}
\end{document}